\documentclass[12pt,a4paper]{article}
\usepackage[utf8]{inputenc}
\usepackage{amsmath}
\usepackage{amsfonts}
\usepackage{amssymb}
\usepackage{amsthm}
\usepackage{hyperref}
\usepackage[left=1.9cm,right=1.9cm,top=2cm,bottom=2cm]{geometry}
\usepackage{xcolor}
\usepackage[all,cmtip]{xy}
\usepackage{tikz}
\usetikzlibrary{decorations.markings}
\usepackage{enumerate}

\newcommand{\U}{\mathrm{U}}

\newcommand{\mft}{\mathfrak{t}}

\newcommand{\RR}{\mathbb{R}}
\newcommand{\CC}{\mathbb{C}}
\newcommand{\ZZ}{\mathbb{Z}}
\newcommand{\PP}{\mathbb{P}}

\newtheorem{thm}{Theorem}[section]
\newtheorem{prop}[thm]{Proposition}
\newtheorem{cor}[thm]{Corollary}
\newtheorem{lem}[thm]{Lemma}

\theoremstyle{definition}
\newtheorem{rem}[thm]{Remark}
\newtheorem{defn}[thm]{Definition}
\newtheorem{ex}[thm]{Example}

\begin{document}

\title{Realization of GKM fibrations and \\ new examples of Hamiltonian non-K\"ahler actions}
\author{Oliver Goertsches\footnote{Philipps-Universit\"at Marburg, email:
goertsch@mathematik.uni-marburg.de}, Panagiotis Konstantis\footnote{Universität zu Köln,
email: pako@mathematik.uni-koeln.de}, and Leopold
Zoller\footnote{Ludwig-Maximilians-Universit\"at M\"unchen, email: leopold.zoller@mathematik.uni-muenchen.de}}

\maketitle
\begin{abstract}
We classify fibrations of abstract $3$-regular GKM graphs over $2$-regular ones, and show that all fiberwise signed fibrations of this type are realized as the projectivization of equivariant complex rank $2$ vector bundles over quasitoric $4$-folds or $S^4$. We investigate the existence of invariant (stable) almost complex, symplectic, and K\"ahler structures on the total space. In this way we obtain infinitely many K\"ahler manifolds with Hamiltonian non-K\"ahler actions in dimension $6$ with prescribed one-skeleton, in particular with prescribed number of isolated fixed points.\\[.5cm]
\end{abstract}

\tableofcontents

\section{Introduction}

While, as shown by Karshon \cite{Karshon}, any effective Hamiltonian circle action on a compact symplectic $4$-fold with finite fixed point set extends to a toric action, an analogous statement in higher dimensions is no longer true. In fact, Tolman \cite{Tolman} gave the first example of a Hamiltonian $T^2$-action on a compact $6$-dimensional symplectic manifold with finite fixed point set that does not admit an invariant K\"ahler structure. Her proof relied solely on the shape of the image of the moment map, or rather the x-ray which also contains the subpolytopes given by the images of the lower-dimensional orbit type strata. In the toric setting, the moment image contains the entire information of the x-ray and by Delzant's theorem \cite{Delzant} there is a one-to-one correspondence between Delzant polytopes and toric manifolds. If one had variants of this correspondence outside of the toric case, this would in theory enable the construction of Hamiltonian non-K\"ahler actions just by drawing specific x-rays. This is the core idea of the present article.

Our language of choice is not the x-ray but the GKM graph \cite{GKM} which encodes the one-skeleton of the space in a labelled graph. In our setting, which is $T^2$-actions in dimension $6$ with finite fixed point set (in particular, these are actions of complexity one \cite{KarshonTolman} which are not tall), this contains, up to lengths of edges, the same information as the x-ray. GKM graphs have the advantage of not being bound to Hamiltonian actions but rather being able to model arbitrary GKM actions while further geometric structures (almost complex, symplectic or K\"ahler) are reflected in properties of the graph (see Section \ref{sec:defsec}). Regarding beginnings of a Delzant-type correspondence, we proved in \cite{1903.11684v1} that in dimension 6, for GKM actions with connected stabilizers on simply-connected manifolds (these conditions are automatic in the toric setting), the GKM graph does encode the non-equivariant diffeomorphism type.
From this we deduced that Tolman's original example is diffeomorphic to Eschenburg's twisted flag manifold ${\mathrm{SU}}(3)/\!/T^2$ \cite{EschenburgHabil, Eschenburg}, which is the projectivization of a complex $T^2$-equivariant rank $2$ vector bundle over $\CC \PP^2$. This implied in particular that Tolman's example is K\"ahler, although of course not in an equivariant fashion.

From the point of view of GKM theory, the fact that Tolman's example is a projectivized equivariant bundle is reflected in the fact that its GKM graph fibers over the GKM graph of $\CC \PP^2$,
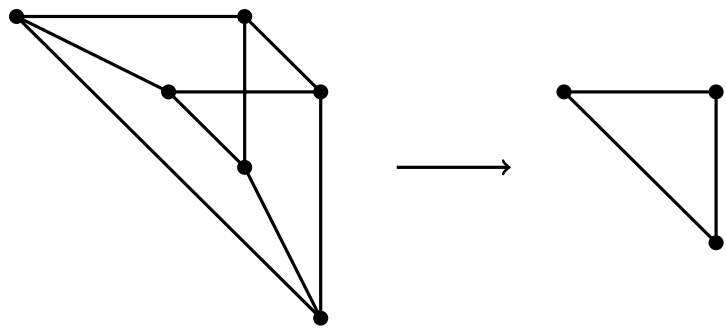
\begin{center}
\begin{tikzpicture}
\draw[very thick] (1,1) -- ++(-3,0) -- ++(4,-4) -- ++(0,3);
\draw[very thick] (2,0) -- ++(-1,1);
\draw[very thick] (2,-3)--++(-1,+2);
\draw[very thick] (1,-1)--++(0,2)++(0,-2)--++(-1,1)--++(2,0);
\draw[very thick] (0,0) -- ++(-2,1);

  \node at (1,1)[circle,fill,inner sep=2pt]{};

  \node at (-2,1)[circle,fill,inner sep=2pt]{};

  \node at (1,-1)[circle,fill,inner sep=2pt]{};

  \node at (0,0)[circle,fill,inner sep=2pt]{};

  \node at (2,0)[circle,fill,inner sep=2pt]{};

  \node at (2,-3)[circle,fill,inner sep=2pt]{};

\draw[->, very thick] (3,-1) -- ++(1.5,0);

\draw[very thick] (5.2,0) -- ++(2,0);
\draw[very thick] (7.2,0) -- ++(0,-2);
\draw[very thick] (7.2,-2) -- ++(-2,2);

  \node at (5.2,0)[circle,fill,inner sep=2pt]{};
  \node at (7.2,0)[circle,fill,inner sep=2pt]{};
  \node at (7.2,-2)[circle,fill,inner sep=2pt]{};

\end{tikzpicture}
\end{center}
see also Example \ref{ex:twflag}. Formally, we make use of the notion of a fibration of abstract GKM graphs, introduced by Guillemin--Sabatini--Zara \cite{GKMFiberBundles}, which we review in Section \ref{sec:GKMfibrations}.

In this paper, we extend this viewpoint on Tolman's example to the more general setting of $6$-dimensional GKM $T^2$-manifolds with an arbitrary (finite) number of fixed points. As stated initially, the purpose here is twofold: on the one hand we contribute to the realization problem of abstract GKM graphs by showing that many GKM fibrations in dimension 6 are in fact realizable. On the other hand we closely investigate how Tolman's original example embeds in this context which uncovers a rich variety of new examples of Hamiltonian non-Kähler actions for any possible number of fixed points.
In a little more detail, our results can be summarized as follows:\\

\noindent {\bf Realization:} We prove that every fiberwise signed GKM fibration with values in $\ZZ^2$ (see Definition \ref{defn:fiberwise signed}) of a 3-regular abstract GKM graph over an effective 2-regular abstract GKM graph can be realized geometrically by a fibration of GKM $T^2$-manifolds (see Theorem \ref{thm:mainthm}). The realization is given as the ($6$-dimensional) projectivization of a rank $2$ complex $T^2$-vector bundle over a $4$-dimensional $T^2$-manifold. The first step to construct the bundle is to do so separately over each invariant two-sphere in the base such that its projectivization is a specific Hirzebruch surface. We proceed to glue those to obtain a bundle over the entire one-skeleton and finally extend the bundle to the whole base with the use of equivariant obstruction theory.\\

\noindent {\bf Geometric Structures:} We go on to show that certain properties of the graphs lead to corresponding geometric structures on the realizations: a signed GKM structure on the base graph lets us choose almost complex realizations and if furthermore the base graph is the boundary of a Delzant polytope, then our realizations are Hamiltonian actions. In the latter case, the $T^2$-invariant symplectic form on the total space also admits a compatible complex structure (in fact the manifold is even projective). However, the complex structure is not necessarily $T^2$-invariant.\\

\noindent {\bf (Non-)Existence of Invariant K\"ahler Structures:} Regarding the $T^2$-invariance of the complex structures which are present in the Hamiltonian case, we prove the following: if the GKM fibration is graph theoretically a M\"obius band and $n-1$ of the $2n$ fixed points ($n\neq 4$) map to the interior of the moment image, then there can not exist a $T^2$-invariant K\"ahler structure (Theorem \ref{thm:signed-classification}). This is done by classifying all possible signed GKM structures that a $T^2$-invariant almost complex structure could induce on the underlying GKM graph and proving individually that they do not come from a K\"ahler action. In combination with the previous results this yields that every fibration of abstract GKM graphs of the the above type gives rise to a $T^2$-action in the spirit of Tolman's original example, i.e.\ a Hamiltonian action with isolated fixed points on a simply-connected compact $6$-dimensional manifold such that no $T^2$-invariant K\"ahler structure exists, while the symplectic form itself is non-equivariantly K\"ahler. Whether the stabilizers of the realizations are connected, as they are in Tolman's example, depends on the respective GKM graph (see Theorem \ref{thm:mainthm}).
Contrary to the Tolman type scenario, in case the GKM fibration is graph theoretically of product type there always exists a $T^2$-invariant complex structure (Section \ref{sec:product type fibrations}).\\

\noindent {\bf Classification of GKM fibrations:} Finally, we quantify the new examples by classifying fiberwise signed GKM fibrations over a fixed
base graph. Up to isomorphism they correspond bijectively to $((\mathbb{Z}-0)^n/\pm)\times
\{0,1\}$ (Proposition \ref{prop:correspondence}). In case $B$ is the boundary of a Delzant polytope
the Hamiltonian non-Kähler actions described above correspond to the elements of the form
$([k_1,\ldots,k_n],1)$ such that $k_i$ and $k_{i-1}$ have the same sign for all but one
$i\in\{1,\ldots,n\}$, where we set $k_0=-k_n$ (Proposition \ref{prop:interior}). In particular every
such tuple in combination with a $2$-dimensional Delzant polytope gives rise to an exotic
Hamiltonian action in the previous sense. We prove that different tuples (up to signs and
permutation) produce examples of different equivariant homotopy type (Cor.\ \ref{cor:distinguish
eq HT}). Additionally we compute the (non-equivariant) cohomology rings as well as the Chern
classes and verify that even in the case of manifolds with $6$ fixed points, which fiber over
$\mathbb{CP}^2$, our method produces infinitely many pairwise not homotopy equivalent examples of
manifolds carrying such an exotic Hamiltonian action (Section \ref{sec:charclasses}).

\paragraph*{Acknowledgements.} We wish to thank Sönke Rollenske for some helpful discussions.

\section{GKM theory and geometric structures}
\label{sec:defsec}

The purpose of this section is to review the basics of GKM theory and see how different geometric structures on manifolds leave their mark on the GKM graph.

\subsection{GKM manifolds}

For an action of a compact torus $T$ on a connected, compact manifold $M$, we consider its fixed point set $M^T=\{p\in M\mid T\cdot p =\{p\}\}$ as well as its \emph{one-skeleton} $M_1=\{p\in M\mid \dim T\cdot p\leq 1\}$. In \emph{GKM theory}, named after Goresky--Kottwitz--MacPherson \cite{GKM} one puts certain assumptions on the action that allow to encode the structure of the one-skeleton in a labelled graph. More precisely, we say that the action \emph{satisfies the GKM conditions} if $M$ is orientable, $M^T$ is a finite set of points, and $M_1$ a finite union of $T$-invariant $2$-spheres.

In this setting, the orbit space of the one-skeleton $M_1/T$ has the structure of a graph $\Gamma$, with one vertex for each fixed point, and one edge for each invariant $2$-sphere. The vertex set of a graph $\Gamma$ will be denoted by $V(\Gamma)$, and the set of edges by $E(\Gamma)$. Formally, we include in the edge set $E(\Gamma)$ of a graph $\Gamma$ each edge twice, once with every possible orientation. For an oriented edge $e\in E(\Gamma)$ we denote its initial vertex by $i(e)$ and its terminal vertex by $t(e)$; the edge $e$ with the opposite orientation will be denoted $\bar{e}$. We write $E(\Gamma)_v$ for the set of edges $e\in E(\Gamma)$ emanating from $v$. At each fixed point $p$ of the action the isotropy representation decomposes into $n$ two-dimensional summands, where $2n$ is the dimension of $M$. The $n$ weights of these irreducible submodules are elements of $\ZZ_\mft^*/\pm 1$, where $\ZZ_\mft^*\subset \mft^*$ is the weight lattice of $T$. Any such weight corresponds uniquely to an invariant $2$-sphere containing $p$, and we put it as a label to the corresponding edge of $\Gamma$. In total, we obtain a map $\alpha:E(\Gamma)\to \ZZ_\mft^*/\pm 1$ which we call an \emph{axial function}, following \cite{GuilleminZaraI}. The graph $\Gamma$, together with the axial function $\alpha$, will be called the \emph{GKM graph} of the $T$-action.

\begin{rem}
Oftentimes, one includes the vanishing of the odd-dimensional cohomology groups of $M$ into the GKM conditions, in order to make the connection between the GKM graph and (equivariant) cohomology. The focus of this paper is the realization of certain GKM graphs and their geometrical properties. All our examples will automatically satisfy this condition.
\end{rem}

Independent of this geometric setting, one can define abstract GKM graphs \cite{GuilleminZaraII}. The graphs one considers have finite vertex and edge sets; we allow multiple edges between vertices, but no loops, i.e., edges that connect a vertex to itself.
\begin{defn}
A \emph{connection} on a graph $\Gamma$ consists of a bijective map $\nabla_e:E(\Gamma)_{i(e)}\to E(\Gamma)_{t(e)}$ for each $e\in E(\Gamma)$, such that
\begin{enumerate}
\item $\nabla_e e = \bar{e}$ and
\item $(\nabla_e)^{-1} = \nabla_{\bar{e}}$ for all $e\in E(\Gamma)$.
\end{enumerate}
\end{defn}
\begin{defn}

An \emph{(abstract) GKM graph} $(\Gamma,\alpha)$ consists of an $n$-valent connected graph $\Gamma$ and a map $\alpha:E(\Gamma)\to \ZZ^m/\pm 1$, called \emph{axial function}, such that there exists a connection $\nabla$ on $\Gamma$ for which the following hold:
\begin{enumerate}
\item For every $v\in V(\Gamma)$ and distinct edges $e,f\in E(\Gamma)_v$ we have that $\alpha(e)$ and $\alpha(f)$ are linearly independent.
\item The connection $\nabla$ is \emph{compatible with $\alpha$}, i.e., for every $v\in V(\Gamma)$ and edges $e,f\in E(\Gamma)_v$ we have
\[
\alpha(\nabla_ef) = \pm \alpha(f) + c\alpha(e)
\]
for some $c\in \ZZ$.
\item For every $e\in E(\Gamma)$ we have $\alpha(\bar{e}) = \alpha(e)$.
\end{enumerate}
\end{defn}
Note that the linear independence of $\alpha(e)$ and $\alpha(f)$ in the first condition, which is defined via preimages under $\ZZ^m\to \ZZ^m/\pm 1$, is well-defined.
\begin{rem}
In the definition of an abstract GKM graph, no acting torus is fixed. Whenever a GKM graph is associated to an action of a torus $T^m$, we identify its Lie algebra $\mft$ with $\RR^m$ in such a way that $\ZZ_\mft^*$ corresponds to $\ZZ^m \subset \RR^m\cong (\RR^m)^*$.
\end{rem}

\begin{rem}
Given an action of a torus $T$ on a connected, compact manifold $M$ satisfying the GKM conditions, the GKM graph of the action admits a compatible connection, see \cite[Proposition 2.3]{GoertschesWiemelerNonNeg}, or \cite{GuilleminZaraII}.
Note that there exist different conventions in the literature of whether the connection is part of the structure of an abstract GKM graph or not. In general, the connection is not unique and there is no canonical choice, which is why we chose to not fix a connection for the abstract object in order to keep the passage from geometry to graphs canonical. Additionally an equivariant diffeomorphism induces a map on graph level (see the notion of isomorphism below) which is not necessarily compatible with specific choices of connections. This is why it is not handy to have connections as part of the abstract data in particular for classification purposes.
\end{rem}

\begin{defn}
Let $(\Gamma,\alpha)$ be a GKM graph and $\Gamma'$ a connected regular subgraph of $\Gamma$. If $(\Gamma',\alpha|_{E(\Gamma')})$ is a GKM graph (i.e.\ it admits a compatible connection) then we call it a GKM subgraph of $(\Gamma,\alpha)$.
\end{defn}

Following \cite{FranzYamanaka} we make the following

\begin{defn}\label{defn:GKM-iso}
An isomorphism $(\Gamma,\alpha)\rightarrow (\Gamma',\alpha')$ between GKM graphs consists of bijections $f$ and $g$ between the vertex and edge sets and an automorphism $\varphi$ of $\mathbb{Z}^m$ such that for any $e\in E(\Gamma)$ we have
\begin{enumerate}
\item $f(i(e))=i(g(e))$
\item $f(t(e))=t(g(e))$
\item $\varphi(\alpha(e))=\alpha'(g(e))$.
\end{enumerate}
\end{defn}

\begin{defn}
We call a GKM graph with labels in $\mathbb{Z}^m$ effective if at one (and hence every) vertex, the labels of the outgoing edges lift to a generating set of $\mathbb{Z}^m$.
\end{defn}

The terminology is justified by the fact that a GKM $T^m$-action on a manifold is effective if and only if the corresponding GKM graph is effective: for any vertex $v$ of the GKM graph, the kernel of the action is given as the intersection of the  kernels of the $\alpha(e)$, where $e$ varies over all edges emanating from $v$. Here, we interpret the $\alpha(e)$ as homomorphisms $T^m\to S^1$.

\subsection{Almost complex structures}

Given a $T$-invariant almost complex structure on $M$, the weights of the isotropy representation have a well-defined sign, i.e., are elements of $\ZZ_\mft^*$. We will speak of the \emph{signed GKM graph} of the action when we consider the graph $\Gamma$ with these weights as labels. Formally, the axial function becomes a map $\alpha:E(\Gamma)\to \ZZ_\mft^*$, by associating to $e$ the weight of the corresponding summand of the isotropy representation at $i(e)$. Abstractly, we define
\begin{defn}
An \emph{abstract signed GKM graph} $(\Gamma,\alpha)$ consists of an $n$-valent connected graph $\Gamma$ and a map $\alpha:E(\Gamma)\to \ZZ^m$, called \emph{axial function},  such that there exists a compatible connection $\nabla$ on $\Gamma$ for which the following hold:
\begin{enumerate}
\item For every $v\in V(\Gamma)$ and distinct edges $e,f\in E(\Gamma)_v$ we have that $\alpha(e)$ and $\alpha(f)$ are linearly independent.
\item The connection $\nabla$ is \emph{compatible with $\alpha$}, i.e., for every $v\in V(\Gamma)$ and edges $e,f\in E(\Gamma)_v$ we have
\[
\alpha(\nabla_ef) = \alpha(f) + c\alpha(e)
\]
for some $c\in \ZZ$.
\item For every $e\in E(\Gamma)$ we have $\alpha(\bar{e}) = -\alpha(e)$.
\end{enumerate}
\end{defn}

\begin{rem}
By composing the axial function of a signed GKM graph with the projection $\mathbb{Z}^m\rightarrow \mathbb{Z}^m/\pm$ one obtains an (unsigned) GKM graph. Note that a connection which is compatible with the signed GKM graph is in particular compatible with the (unsigned) GKM graph. We also call the signed graph a \emph{compatible signed structure} of the underlying GKM graph. There might be different signed structures compatible with a single GKM graph, corresponding to the existence of different homotopy classes of almost complex structures compatible with the action. The existence of a compatible signed structure is clearly an obstruction to the existence of an invariant almost complex structure. E.g.\ it is easy to check that the GKM graph
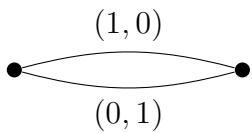
\begin{center}
\begin{tikzpicture}

\node (a) at (0,0)[circle,fill,inner sep=2pt] {};
\node (b) at (3,0)[circle,fill,inner sep=2pt]{};
\node at (1.5,0.6) {$(1,0)$};
\node at (1.5,-0.6) {$(0,1)$};

\draw (a) to[very thick, in=165, out=15, ] (b);
\draw (a) to[very thick, in=195, out=-15] (b);

\end{tikzpicture}
\end{center}
of the standard $T^2$-action on $S^4\subset \mathbb{C}^2\oplus \mathbb{R}$ does not carry a compatible signed structure.
\end{rem}

The definitions of GKM subgraphs and isomorphisms of GKM graphs carry over to the signed setting in an obvious fashion.

\subsection{Symplectic structures}\label{subsec:introsymplec}

Just as the existence of complex structures impacts the GKM graph of a manifold, the existence of a compatible symplectic structure forces certain properties onto the GKM graph. More precisely, we consider for a Hamiltonian GKM action of a torus $T$ on a manifold $M$ a moment map $\mu\colon M\rightarrow \mathfrak{t}^*$ into the dual of the Lie algebra of $T$. An invariant two sphere $S^2\subset M$ gets mapped under $\mu$ to an affine linear interval in $\mathfrak{t}^*$ whose boundary points are the images of the fixed points. By piecing those together we obtain a linear realisation of the underlying graph of the GKM graph, i.e.\ an edge-wise affine linear map from the topological realization of $\Gamma$ to $\mathfrak{t}^*$ (not an embedding!), whose image is $\mu(M_1)$. If $(\Gamma,\alpha)$ is the signed GKM graph associated to an almost complex structure which is compatible with the symplectic form, then the labels are encoded in the linear realization as follows: if $e$ is an oriented edge, then the slope of the corresponding affine linear segment in $\mathfrak{t}^*$ is given by $\alpha(e)\in \mathbb{Z}_\mft^*\subset \mft^*$. Thus the linear realization determines $(\Gamma,\alpha)$ up to multiples of the weights. The labels are uniquely determined if we add the assumption that they are primitive.

\begin{rem}\label{rem:symplecticobstruction}
Having such a linear realisation is a non-trivial obstruction for a signed GKM graph to come from a Hamiltonian action. For example the signed structures of type III in Theorem \ref{thm:signed-classification} do not since the cone spanned by the weights of the outgoing edges at every point is all of $\mft^*$.

The convexity theorem, which implies that the image $\mu(M)$ of the moment map is the convex hull of $\mu(M_1)$, gives additional obstructions. The preimage of an outer edge of the resulting polytope is contained in $M_1$. Thus every outer edge of the convex hull $\mu(M)$ must be an edge in $\mu(M_1)$. As a counterexample, the linear realization

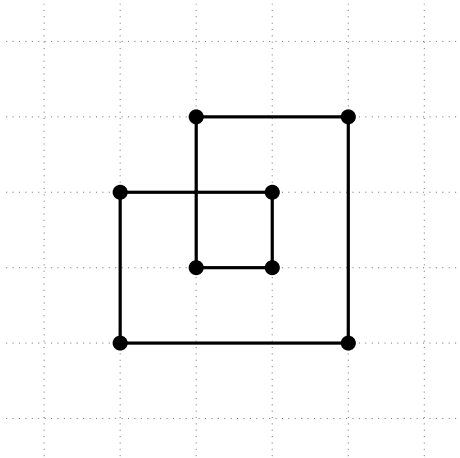
\begin{center}
\begin{tikzpicture}
\draw[step=1, dotted, gray] (-3.5,-3.5) grid (2.5,2.5);

\draw[very thick] (-2,0) -- ++(2,0) -- ++(0,-1) -- ++ (-1,0) -- ++(0,2) -- ++ (2,0) --++(0,-3)-- ++(-3,0) -- ++ (0,2);

  \node at (-2,0)[circle,fill,inner sep=2pt]{};

  \node at (-1,1)[circle,fill,inner sep=2pt]{};

  \node at (1,1)[circle,fill,inner sep=2pt]{};

  \node at (1,-2)[circle,fill,inner sep=2pt]{};

  \node at (-2,-2)[circle,fill,inner sep=2pt]{};

  \node at (0,-1)[circle,fill,inner sep=2pt]{};
  \node at (-1,-1)[circle,fill,inner sep=2pt]{};

  \node at (0,0)[circle,fill,inner sep=2pt]{};
\end{tikzpicture}

\end{center}
does not have this property (the grid represents the standard basis of $\mathbb{Z}^2$) and it follows from the results in Section \ref{subsec:case II} that, more generally, the signed GKM graph which is uniquely defined by having this realization and primitive weights does not admit any linear realization with the above convexity property. Hence it can not come from a Hamiltonian action.
\end{rem}

Note that the convexity obstruction from the above remark does not concern the inner edges of the linear realization. However, in order for a GKM graph to come from a Hamiltonian $T$-action additional obstructions for inner edges do exist: for any subtorus $T'\subset T$, the action of $T$ on the submanifold $M^{T'}$ is again Hamiltonian, so certain subgraphs of the original GKM graph do again satisfy the convexity criterion explained in Remark \ref{rem:symplecticobstruction}. These subtleties will not play a role in our low dimensional considerations so we refrain from introducing a precise definition outside of the case below. Note that GKM graphs of $4$-dimensional Hamiltonian manifolds do indeed satisfy the following

\begin{defn}
We say that a signed $2$-regular GKM graph with labels in $\mathbb{Z}^m$ is of polytope type if, graph theoretically, it is given by the edges of a simple convex $2$-polytope in $\mathbb{R}^m$ and the labels of the oriented edges are given by integral representatives of the slopes of the edges.
\end{defn}

\subsection{K\"ahler structures}

As a last step in the hierarchy we state a certain property of signed GKM graphs coming from invariant K\"ahler structures which goes beyond the previous obstructions for Hamiltonian actions. In \cite[Lemma 3.5]{Tolman} the following is shown: consider a Hamiltonian $T$-action on $(M,\omega)$ with $\omega$ K\"ahler and a collection $V\subset T_pM$ of irreducible summands of the isotropy representation at a fixed point $p$ such that the weights of those summands form a minimal generating set of a convex cone in $\mft^*$. Then there is a Hamiltonian submanifold of $M$ containing $p$ whose tangent space at $p$ is $V$. As a special case, this implies the following

\begin{cor}\label{cor:Tolmanextension}
For any pair of adjacent edges in the signed GKM graph associated to a Hamiltonian action of GKM type on a Kähler manifold, there is a $2$-regular GKM subgraph of polytope type containing those edges.
\end{cor}

An example where this fails while the obstructions for general Hamiltonian actions hold is given by the GKM graph with primitive labels and the linear realization

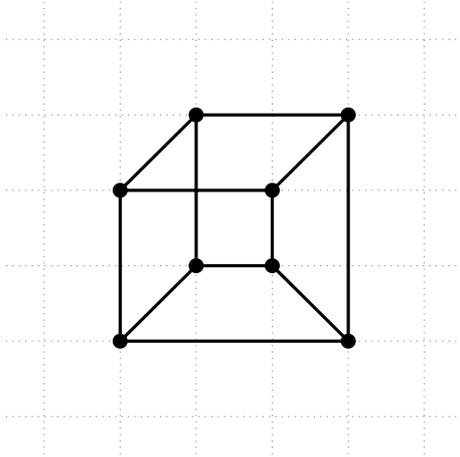
\begin{center}
\begin{tikzpicture}
\draw[step=1, dotted, gray] (-3.5,-3.5) grid (2.5,2.5);

\draw[very thick] (-2,0) -- ++(2,0) -- ++(0,-1) -- ++ (-1,0) -- ++(0,2) -- ++ (2,0) --++(0,-3)-- ++(-3,0) -- ++ (0,2);

\draw[very thick] (-2,0) -- ++(1,1);
\draw[very thick] (0,0) -- ++(1,1);
\draw[very thick] (-2,-2) -- ++(1,1);
\draw[very thick] (1,-2) -- ++(-1,1);

  \node at (-2,0)[circle,fill,inner sep=2pt]{};

  \node at (-1,1)[circle,fill,inner sep=2pt]{};

  \node at (1,1)[circle,fill,inner sep=2pt]{};

  \node at (1,-2)[circle,fill,inner sep=2pt]{};

  \node at (-2,-2)[circle,fill,inner sep=2pt]{};

  \node at (0,-1)[circle,fill,inner sep=2pt]{};
  \node at (-1,-1)[circle,fill,inner sep=2pt]{};

  \node at (0,0)[circle,fill,inner sep=2pt]{};
\end{tikzpicture}

\end{center}
since the only signed GKM subgraph containing the edges between the three inner vertices is the graph from Remark \ref{rem:symplecticobstruction}. See also Section \ref{subsec:case I} where we will consider this obstruction to the K\"ahler property for a large class of graphs.

\section{GKM fibrations}
\label{sec:GKMfibrations}

Let us review the definition of a GKM fibration, introduced by Guillemin--Sabatini--Zara in \cite{GKMFiberBundles}.

A \emph{morphism} of graphs $\pi:\Gamma\to B$ consists of a map sending vertices of $\Gamma$ to vertices of $B$, as well as a map sending an edge between vertices $p,q\in V(\Gamma)$ with $\pi(p)\neq \pi(q)$ to an edge between $\pi(p)$ and $\pi(q)$. Edges in $\Gamma$ between $p,q\in V(\Gamma)$ with $\pi(p)=\pi(q)$ are called \emph{vertical}; the other edges are called \emph{horizontal}. For $p\in V(\Gamma)$ the set of horizontal edges emanating from $p$ is denoted by $H_p$.
\begin{rem}
In \cite{GKMFiberBundles} morphisms of graphs are defined only on vertices, not on edges. As we allow multiple edges between vertices, we need to specify images of edges as well.
\end{rem}
The morphism $\pi$ is a \emph{graph fibration} if for all $p\in V(\Gamma)$ the map $\pi:H_p\longrightarrow E(B)_{\pi(p)}$ is a bijection. In other words, graph fibrations have a unique path-lifting property: given a base vertex $p\in \Gamma$ and an edge $e\in E(B)$ with $i(e)=\pi(p)$, there exists a unique horizontal edge which lies over $e$ and starts at $p$.

Of course a fibration of GKM graphs should be compatible with the additional structure. There are analogous versions of this notion for the signed and the unsigned case:
\begin{defn}\label{defn:GKMfibration}Let $(\Gamma,\alpha)$ and $(B,\alpha_B)$ be (signed) GKM graphs.
A graph fibration $\pi:\Gamma\to B$ is a \emph{(signed) GKM fibration} if there exist connections $\nabla$ and $\nabla^B$ which are compatible with the (signed) GKM structures on $\Gamma$ and $B$ such that additionally the following hold:
\begin{enumerate}
\item For any edge $e$ of $B$ and any lift $\tilde{e}$ of $e$ we have $\alpha_B(e) = \alpha(\tilde{e})$.
\item For every edge $e$ of $\Gamma$, the connection $\nabla_e$ sends vertical edges to vertical edges (and thus horizontal edges to horizontal edges).
\item For two edges $e,e'$ of $B$ with $i(e) = i(e')$ and lifts $\tilde{e}, \tilde{e}'$ of $e$ and $e'$ with $i(\tilde{e})=i(\tilde{e'})$, the edge $\nabla_{\tilde{e}} \tilde{e}'$ is the lift of $(\nabla^B)_e e'$ at $t(\tilde{e})$.
\end{enumerate}
\end{defn}

\begin{rem}
Note that our definition deviates from that in \cite{GKMFiberBundles} in that we do not fix connections as part of the data of GKM fibrations.
For the sake of completeness we also note that there is the stronger notion of a GKM fiber bundle, which was introduced in \cite{GKMFiberBundles} in the signed case. However, our main interest in this article lies in dimension $6$ and the GKM fibrations we consider will automatically fulfil the stronger requirements of GKM fiber bundles. Thus, there is no need for us to introduce this more restrictive notion.
\end{rem}

The GKM fibrations we can realize geometrically through our main result will have almost complex fibers. However the base will not need to have an almost complex structure. The natural setting for this is given by the following definition which is an intermediate notion between unsigned and signed GKM fibrations.

\begin{defn}\label{defn:fiberwise signed}
Let $\pi\colon(\Gamma,\alpha)\rightarrow (B,\alpha_B)$ be a GKM fibration (of unsigned graphs). Let $F\subset E(\Gamma)$ be the set of vertical edges and $\tilde{\alpha}\colon F\rightarrow \mathbb{Z}^m$ a lift of $\alpha\colon E(\Gamma)\rightarrow \mathbb{Z}^m/\pm$ satisfying $\tilde{\alpha}(e)=-\tilde{\alpha}(\overline{e})$. Then we call $\pi$ together with $\tilde{\alpha}$ a \emph{fiberwise signed fibration} if the connections $\nabla$ and $\nabla^B$ as in Definition \ref{defn:GKMfibration} can be chosen in a way such that $\tilde{\alpha}(\nabla_{e}e')\equiv \tilde{\alpha}(e')\mod \alpha(e)$ for any $e'\in F$ and $e\in E(\Gamma)$.
\end{defn}

\begin{lem}\label{lem:fiberwise signed and signed}
Every signed fibration of signed GKM graphs gives rise to a fiberwise signed fibration of the underlying GKM graphs. Conversely if $(\pi,\tilde{\alpha})$ is a fiberwise signed fibration as above, then any signed structure $(B,\tilde{\alpha}_B)$ compatible with the base graph gives rise to a unique signed structure on $\Gamma$ which extends $\tilde{\alpha}$ such that $\pi$ becomes a signed fibration.
\end{lem}

\begin{proof}
The first statement is clear. For the second statement, note that an extension of $\tilde{\alpha}$ to $E(\Gamma)$ such that $\pi\colon (\Gamma,\tilde{\alpha})\rightarrow  (B,\tilde{\alpha}_B)$ is a signed fibration is unique: on every horizontal edge $e\in E(\Gamma)$ we need to define $\tilde{\alpha}(e)=\tilde{\alpha}_B(\pi(e))$. So it remains to check the existence of a compatible connection. Let $\nabla$ be a connection on $\Gamma$ as in Definition \ref{defn:fiberwise signed} and $\nabla^B$ be a connection compatible with the signed graph $(B,\tilde{\alpha}_B)$. We define a new connection $\nabla'$ as follows: For any edge $e\in E(\Gamma)$ we set $\nabla'_e$ as $\nabla'_e(e')=\nabla_e(e')$ if $e'$ is vertical. On horizontal edges, we define $\nabla'_e$ as
\[H_{i(e)}\xrightarrow{\pi} B_{i(\pi(e))}\xrightarrow{\nabla^B_\pi(e)} B_{t(\pi(e))}\xrightarrow{\pi^{-1}} H_{t(e)}\]
if $e$ is horizontal and as
\[H_{i(e)}\xrightarrow{\pi} B_{i(\pi(e))}= B_{t(\pi(e))}\xrightarrow{\pi^{-1}} H_{t(e)}\]
if $e$ is vertical. Then $\nabla'$ and $\nabla^B$ satisfy the requirements for connections compatible with signed fibrations.
\end{proof}

\begin{defn}
We call two (signed) GKM fibrations $\pi\colon (\Gamma,\alpha)\rightarrow (B,\alpha_B)$ and $\pi'\colon (\Gamma',\alpha')\rightarrow (B,\alpha_B)$ equivalent if there is an isomorphism $(f,g,\varphi)\colon (\Gamma,\alpha)\rightarrow (\Gamma',\alpha')$
of (signed) GKM graphs as in Definition \ref{defn:GKM-iso} with $\varphi=\mathrm{id}_{\mathbb{Z}^m}$, which respects the decomposition into vertical and horizontal edges and commutes with the fibrations on vertices and horizontal edges.
Two fiberwise signed fibrations $(\pi,\tilde{\alpha})$ and $(\pi',\tilde{\alpha}')$ are called equivalent if there is an equivalence $(f,g,{\mathrm{id}_\mathbb{Z}^m})$ of the underlying GKM fibrations such that additionally $\tilde{\alpha}'(g(e))=\tilde{\alpha}(e)$ for every vertical edge $e\in E(\Gamma)$.
\end{defn}

\section{GKM fibrations in dimension 6}\label{sec:graphs in dim 6}

In this section we consider GKM fibrations $\Gamma\rightarrow B$ where $\Gamma$ is 3-regular, $B$ is $2$-regular and weights (up to sign) are in $\mathbb{Z}^2$ (corresponding to an equivariant fibration of a 6-dimensional $T^2$-manifold over a 4-dimensional $T^2$-manifold). All fibrations will be assumed to be of this form even if not explicitly stated. Note that graph theoretically there is not much variety to what can happen: $B$ is necessarily an $n$-gon (since it is 2-regular and connected).

\begin{defn}
If the lifts of a path around the $n$-gon $B$ are closed in $\Gamma$, then we say $\Gamma$ is of product type. If not, then we say $\Gamma$ is of twisted type.
\end{defn}

\begin{ex}
The following are examples of linear realizations of total spaces of GKM fibrations of twisted type over a $5$- respectively $6$-gon.

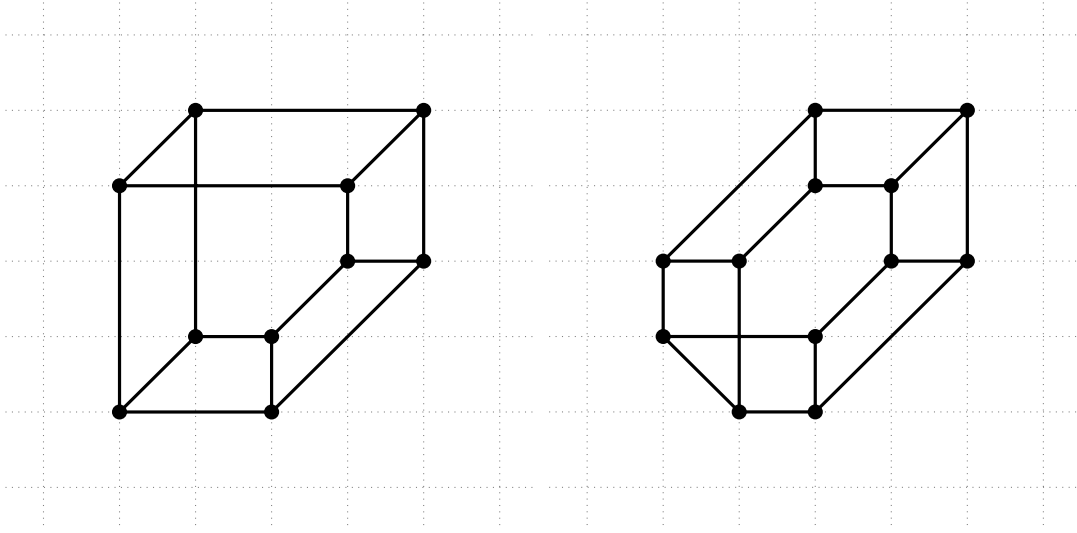
\begin{center}
\begin{tikzpicture}
\draw[step=1, dotted, gray] (-3.5,-4.5) grid (3.5,2.5);

\draw[very thick] (-2,0) -- ++(1,1) -- ++(3,0) -- ++(0,-2) -- ++(-2,-2) -- ++(-2,0) -- ++(0,3);
\draw[very thick] (-2,0) -- ++(3,0) -- ++(0,-1) -- ++(-1,-1) -- ++(-1,0) -- +(0,3);
\draw[very thick] (-2,-3)--++(1,1);
\draw[very thick] (1,0)--++(1,1);
\draw[very thick] (1,-1)--++(1,0);

\draw[very thick] (0,-2)--++(0,-1);

  \node at (-2,0)[circle,fill,inner sep=2pt]{};

  \node at (-1,1)[circle,fill,inner sep=2pt]{};

  \node at (2,1)[circle,fill,inner sep=2pt]{};

  \node at (2,-1)[circle,fill,inner sep=2pt]{};

  \node at (0,-3)[circle,fill,inner sep=2pt]{};

  \node at (-2,-3)[circle,fill,inner sep=2pt]{};
  \node at (-1,-2)[circle,fill,inner sep=2pt]{};

  \node at (0,-2)[circle,fill,inner sep=2pt]{};
  \node at (1,-1)[circle,fill,inner sep=2pt]{};
  \node at (1,0)[circle,fill,inner sep=2pt]{};

\end{tikzpicture}
\begin{tikzpicture}
\draw[step=1, dotted, gray] (-3.5,-3.5) grid (3.5,3.5);

\draw[very thick] (-2,0) -- ++(2,2) -- ++(2,0) -- ++(0,-2) -- ++(-2,-2) -- ++(-1,0) -- ++(-1,1) -- ++(0,1);
\draw[very thick] (-2,-1) -- ++(2,0) -- ++(1,1) -- ++(0,1) -- ++(-1,0) -- ++(-1,-1) -- ++(0,-2);
\draw[very thick] (-2,0) -- ++(1,0);
\draw[very thick] (0,1) -- ++(0,1);
\draw[very thick] (1,1) -- ++(1,1);
\draw[very thick] (1,0) -- ++(1,0);
\draw[very thick] (0,-1) -- ++(0,-1);

  \node at (-2,0)[circle,fill,inner sep=2pt]{};

  \node at (0,2)[circle,fill,inner sep=2pt]{};

  \node at (2,2)[circle,fill,inner sep=2pt]{};

  \node at (2,0)[circle,fill,inner sep=2pt]{};

  \node at (0,-2)[circle,fill,inner sep=2pt]{};

  \node at (-1,-2)[circle,fill,inner sep=2pt]{};
  \node at (-2,-1)[circle,fill,inner sep=2pt]{};

  \node at (0,-1)[circle,fill,inner sep=2pt]{};
  \node at (1,0)[circle,fill,inner sep=2pt]{};
  \node at (1,1)[circle,fill,inner sep=2pt]{};

  \node at (0,1)[circle,fill,inner sep=2pt]{};
  \node at (-1,0)[circle,fill,inner sep=2pt]{};

\end{tikzpicture}

\end{center}
\end{ex}

It is not hard to see that the underlying graph of $\Gamma$ is determined up to isomorphism by whether it is of product or of twisted type: it either looks like a M\"obius band or the product of a circle with an interval. However if we add the additional structure of the labels to the picture, the situation becomes more interesting.

\begin{lem}\label{lem:different signed structures}
If $\pi\colon (\Gamma,\alpha)\rightarrow (B,\alpha_B)$ admits a compatible structure
of a fiberwise signed fibration, then there are precisely two possible choices for the lift $\tilde{\alpha}\colon F\rightarrow \mathbb{Z}^2$ of $\alpha$ on vertical edges. Both choices are equivalent as fiberwise signed fibrations.
\end{lem}

\begin{proof}
Suppose we have two lifts $\tilde{\alpha},\tilde{\alpha}'\colon F\rightarrow \mathbb{Z}^2$ of $\alpha$ where $F\subset E(\Gamma)$ are the vertical edges. Let $e\in E(\Gamma)$ be a horizontal edge. If $e'$ and $e''$ are the unique vertical edges emanating from $i(e)$ and $t(e)$ then a compatible connection $\nabla$ necessarily satisfies $\nabla_e(e')=e''$. Thus $\tilde{\alpha}(e')\equiv \tilde{\alpha}(e'')\mod \alpha(e)$ and $\tilde{\alpha}'(e')\equiv \tilde{\alpha}'(e'')\mod \alpha(e)$. It follows that $\tilde{\alpha}$ and $\tilde{\alpha}'$ either agree or disagree on both, $e'$ and $e''$. Inductively, this extends to all vertical edges. Conversely if $\tilde{\alpha}$ defines a fiberwise signed structure, then $-\tilde{\alpha}$ clearly does as well. An equivalence of $(\pi,\tilde{\alpha})$ and $(\pi,-\tilde{\alpha})$ is given by the isomorphism that interchanges the vertices in each fiber.
\end{proof}

It follows from the lemma above that two fiberwise signed fibrations are equivalent as such if and only if they are equivalent as unsigned GKM fibrations. Thus  equivalence classes of fiberwise signed fibrations naturally form a subset of equivalence classes of (unsigned) GKM fibrations.

\begin{rem} There is an involution on the set of equivalence classes of GKM fibrations:
given $\Gamma\rightarrow B$, choose two basic edges covering the same edge in the base graph. In one of the fibers we detach the edges from their vertex and reglue them but with the fiber vertices interchanged. The labels of the edges stay the same.
This construction is of course self inverse. It maps GKM fibrations which admit the structure of a fiberwise signed fibration to GKM fibrations which do not carry such a structure and vice versa.
\end{rem}

In the light of the above remark, a classification of fiberwise signed fibrations extends to a classification of all GKM fibrations. In view of Lemma \ref{lem:fiberwise signed and signed} it also extends to a classification of signed fibrations over a signed base graph.

\begin{prop}\label{prop:correspondence}
Let $B$ be an effective $2$-valent GKM graph with $n$ vertices. Then there is a bijective correspondence
\[\text{fiberwise signed 3-valent GKM fibrations over B}/\sim\quad\longleftrightarrow\quad ((\mathbb{Z}-0)^n/\pm)\times\{0,1\},\]
where $\sim$ denotes equivalence of fibrations. For a fixed signed structure on $B$ this induces a bijection
\[\text{signed 3-valent GKM fibrations over B}/\sim\quad\longleftrightarrow\quad ((\mathbb{Z}-0)^n/\pm)\times\{0,1\}.\]

\end{prop}

\begin{rem}\label{rem:correspondencedata}
The above correspondence is not canonical and depends on a fixed choice of data in the GKM graph $B$ which we state here separately for later reference. Let $v_1,\ldots,v_n$ be the vertices of $B$ and $e_1,\ldots,e_n$ its edges, where $e_i$ goes from $v_i$ to $v_{i+1}$. We extend the notation for all $i\in\mathbb{Z}$ by setting $v_{i+n}:=v_i$ and $e_{i+n}=e_i$.

If $B$ comes with a signed structure, then we have unique signs for the weights $\gamma_i$ associated to the $e_i$ and we use these to define the correspondence. Without the signed structure there are choices to make:
let $\gamma_1,\gamma_2\in \mathbb{Z}_\mathfrak{t}^*$ be the weights associated to $e_1$ and $e_2$ in $B$, where we choose the signs arbitrarily. Now we choose representatives for the weights $\gamma_i\in \mathbb{Z}_\mft^*$ of all $e_i$ with the unique sign such that\[\gamma_i\equiv -\gamma_{i+2}\mod \gamma_{i+1}.\]
This is possible thanks to the existence of a compatible connection. Again, it turns out handy to extend the notation for all $i\in\mathbb{Z}$ such that $\gamma_i$ and $\gamma_{i+n}$ correspond to the same edge and thus agree up to sign. These will occasionally play a role and will be denoted through the equation $\gamma_i=(-1)^{\varepsilon_i}\gamma_{i+n}$. Note that the value of $\varepsilon_i$ only depends on whether $i$ is even or odd. The $\varepsilon_i$ vanish if and only if the $\gamma_i$ come from a signed compatible structure.

We will define the correspondence explicitly in the course of the proof; briefly, we map
\[
(\pi:\Gamma\to B)\longmapsto ([k_1,\ldots,k_n],\eta),
\]
where $\eta\in \{0,1\}$ describes if $\pi$ is of product or of twisted type, and the numbers $k_i$ are determined by Equation \eqref{eq:correspondencefiberwiseinbasis}, i.e., given as the coefficients of the expansion of the fiber weights in the bases given by the weights of the adjacent horizontal edges.
\end{rem}

\begin{proof}[Proof of Proposition \ref{prop:correspondence}] Note first that the statement on signed fibrations follows directly from the statement on fiberwise signed fibrations with the help of Lemma \ref{lem:fiberwise signed and signed}.

We begin by associating an element on the right hand side to a fiberwise signed GKM fibration $\Gamma\rightarrow B$. The $\{0,1\}$ component is determined by the graph structure of $\Gamma$: we set it to be $0$ if $\Gamma$ is of product type and $1$ if it is of twisted type. The association of the $(\mathbb{Z}-0)^n/\pm$ component depends on the fixed data from Remark \ref{rem:correspondencedata}.
Now choose an orientation of the edge in the fiber over $v_1$ and let $\alpha_1\in \mathbb{Z}_\mft^*$ be the associated weight (with unique sign). A compatible connection allows us to inductively choose orientations for the fiber edges over $v_i$ in a compatible way such that the associated weights satisfy
\[\alpha_i\equiv \alpha_{i+1} \mod \gamma_i,\]
for $i\in\mathbb{Z}$. Note however that if $\Gamma$ is of twisted type, then transporting a vertical edge around $\Gamma$ once reverses its orientation and thus the orientation used for the definition of $\alpha_i$ might differ from the one of $\alpha_{i+n}$. We have $\alpha_i=(-1)^\eta\alpha_{i+n}$.

By assumption, the weights of two adjacent edges in $B$ form a basis of $\mathbb{Z}_\mft^*$. Thus, for $i\in {\mathbb{Z}}$, there are unique integers $k_i,l_i$ such that $\alpha_i=k_{i}\gamma_{i-1}+l_{i}\gamma_i$. We claim that $k_i=-l_{i+1}$. Transporting the vertical edges along the horizontal ones we find integers $d_i,d_i'$, $i=1,\ldots,n$, such that $\alpha_{i+1}=\alpha_i+d_i\gamma_i$ and $\gamma_{i+1}=-\gamma_{i-1}+d_i'\gamma_i$. We obtain $\alpha_i=(k_{i+1}-d_i+l_{i+1}d_i')\gamma_i-l_{i+1}\gamma_{i-1}$. Uniqueness of the $k_i$ and $l_i$ yields the claim. In particular we have
\begin{equation} \label{eq:correspondencefiberwiseinbasis}
\alpha_i= k_i\gamma_{i-1}-k_{i-1}\gamma_i.
\end{equation}
for $i\in {\mathbb{Z}}$. The $(\mathbb{Z}-0)^n/\pm$ component on the right hand side of the correspondence is now defined by the equivalence class of $(k_1,\ldots,k_n)$. Recall that in the construction of the $k_i$ we made a choice for the orientation of the edge over $v_1$ giving rise to $\alpha_1$. A different choice would lead to a global sign change for the $k_i$ so we obtain a well defined element of $(\mathbb{Z}-0)^n/\pm$.
This association is easily seen to be invariant under equivalences of GKM fibrations.

Conversely we check that the construction can be reversed. Given an element on the right hand side of the correspondence, choose a representative $(k_1,\ldots,k_n,\eta)\in(\mathbb{Z}-0)^n\times\{0,1\}$.
If $\eta=1$ let $\Gamma$ be the unique $3$-regular abstract graph of twisted type that fibers over $B$. Otherwise let $\Gamma$ be the unique such graph of product type. We need to construct labels for the edges of $\Gamma$ that turn the graph fibration $\pi:\Gamma\to B$ into a GKM fibration. To do this set $k_0=(-1)^{\varepsilon_1+\eta} k_n$, where $\varepsilon_1$ is defined as in Remark \ref{rem:correspondencedata}. Then we assign labels to $\Gamma$ as follows: the basic edges over $e_i$ are labelled by $\gamma_i$. For the fiber edges, consider a lift of a path that goes around the $n$-gon $B$ once and let $p_i$ be the vertex in that path which lies over $v_i$. To the directed fiber edge emanating from $p_i$ we associate the weight
\[\alpha_i= k_i\gamma_{i-1}-k_{i-1}\gamma_i\]
for $i=1,\ldots,n$. Let $\nabla^B$ be the unique connection on $B$ and define a compatible connection $\nabla$ on $\Gamma$ as follows: if $e\in E(\Gamma)$ is vertical, then $\nabla_e$ is defined as
\[H_{i(e)}\xrightarrow{\pi} B_{i(\pi(e))}= B_{t(\pi(e))}\xrightarrow{\pi^{-1}} H_{t(e)}.\]
Transport along horizontal edges is uniquely defined by the condition that it respects vertical and horizontal edges.
The connections $\nabla$ and $\nabla^B$ are easily seen to be compatible with the labels, the most interesting step being to verify that for a lift $\tilde{e}_n$ of $e_n$ the connection $\nabla_{\tilde{e}_n}$ satisfies the congruence relations for the labels. To do this recall that the orientation of an edge, when transported along the lift of a path around $B$, gets reversed if and only if $\Gamma$ is of twisted type. Thus we need to have $\alpha_n=(-1)^\eta\alpha_1\mod\gamma_n$. The left hand side however is given by \[\alpha_n = k_n\gamma_{n-1}-k_{n-1}\gamma_n\equiv -k_n\gamma_{n+1}\equiv(-1)^{\eta+1} k_0\gamma_1\equiv(-1)^\eta\alpha_1\mod\gamma_n.\]
The connections are also clearly compatible with the fibration.

The equivalence class does not depend on the chosen lift of the path around $B$ as the other lift will result in the same labels but with a global sign change for the $\alpha_i$. By Lemma \ref{lem:different signed structures} these two fibrations are equivalent. The same is accomplished by a global sign change of the $k_i$ so the construction factors through $((\mathbb{Z}-0)^n/\pm)\times\{0,1\}$.
\end{proof}

\tikzset{middlearrow/.style={
        decoration={markings,
            mark= at position 0.5 with {\arrow{#1}} ,
        },
        postaction={decorate}
    }
}

\begin{ex}\label{ex:flag} Consider the $T^2$-equivariant $\CC \PP^1$-fibration ${\mathrm{SU}}(3)/T^2\to {\mathrm{SU}}(3)/{\mathrm{S}}({\mathrm{U}}(2)\times {\mathrm{U}}(1))=\CC \PP^2$. Its GKM fibration is as follows, see \cite[Section 2.1]{GKMFiberBundles}: \\[-.2cm]
\begin{center}
\begin{tikzpicture}

\draw[very thick] (1,1) -- ++(-3,0) -- ++(4,-4) -- ++(0,3);
\draw[middlearrow={>}, very thick] (2,0) -- ++(-1,1);
\draw[middlearrow={>},very thick] (2,-3)--++(-1,0);
\draw[very thick] (1,-3)--++(0,4);
\draw[very thick] (1,-3) -- ++(-3,3) -- ++ (4,0);
\draw[middlearrow={>},very thick] (-2,0) -- ++(0,1);
\node at (-2.4,.5){$\alpha_1$};
\node at (1.8,.7){$\alpha_2$};
\node at (1.5,-3.35) {$\alpha_3$};
  \node at (1,1)[circle,fill,inner sep=2pt]{};

  \node at (-2,1)[circle,fill,inner sep=2pt]{};

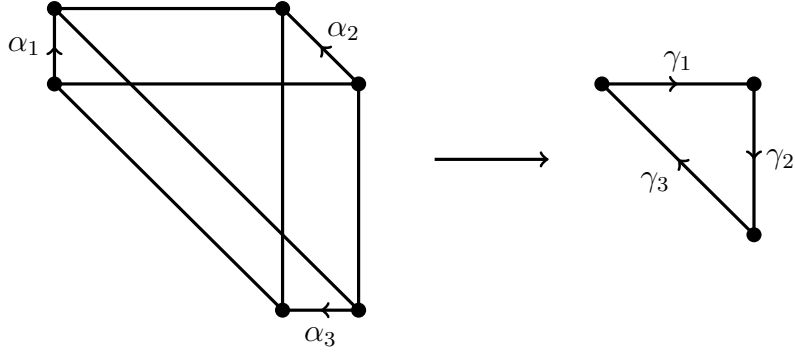
  \node at (1,-3)[circle,fill,inner sep=2pt]{};

  \node at (-2,0)[circle,fill,inner sep=2pt]{};

  \node at (2,0)[circle,fill,inner sep=2pt]{};

  \node at (2,-3)[circle,fill,inner sep=2pt]{};

\draw[->, very thick] (3,-1) -- ++(1.5,0);

\draw[middlearrow={>}, very thick] (5.2,0) -- ++(2,0);
\draw[middlearrow={>}, very thick] (7.2,0) -- ++(0,-2);
\draw[middlearrow={>}, very thick] (7.2,-2) -- ++(-2,2);

  \node at (5.2,0)[circle,fill,inner sep=2pt]{};
  \node at (7.2,0)[circle,fill,inner sep=2pt]{};
  \node at (7.2,-2)[circle,fill,inner sep=2pt]{};
  \node at (6.2,.3){$\gamma_1$};
  \node at (7.55,-1){$\gamma_2$};
  \node at (5.9,-1.3){$\gamma_3$};

\end{tikzpicture}
\end{center}
With the indicated fiber orientation, we have $\alpha_1 = \gamma_3 + \gamma_1$, $\alpha_2 = -\gamma_1 - \gamma_2$, and $\alpha_3 = \gamma_2 + \gamma_3$. Thus this fibration corresponds to $([k_1,k_2,k_3],\eta) = ([1,-1,1],1)$.
\end{ex}

\begin{ex}\label{ex:twflag}
Tolman's example \cite{Tolman}, Woodward's variant \cite{Woodward}, and Eschen\-burg's twisted flag manifold ${\mathrm{SU}}(3)//T^2$ \cite{EschenburgHabil, Eschenburg, 1812.09689v1}, which are (non-equivariantly) diffeomorphic by \cite{1903.11684v1},  fiber equivariantly over $\CC \PP^2$. Their associated GKM fibrations are the following:\\[-.2cm]
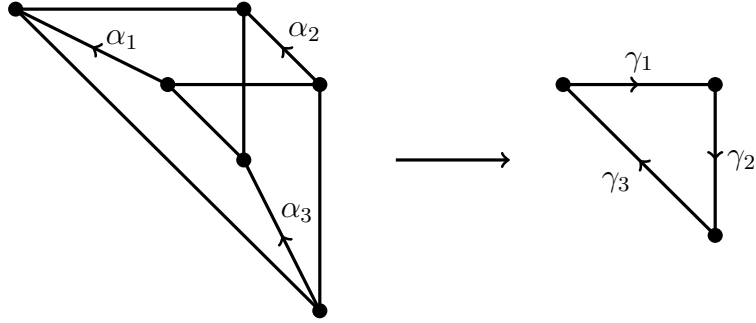
\begin{center}
\begin{tikzpicture}
%\draw[step=1, dotted, gray] (-3.5,-4.5) grid (3.5,2.5);

\draw[very thick] (1,1) -- ++(-3,0) -- ++(4,-4) -- ++(0,3);
\draw[middlearrow={>}, very thick] (2,0) -- ++(-1,1);
\draw[middlearrow={>},very thick] (2,-3)--++(-1,+2);
\draw[very thick] (1,-1)--++(0,2)++(0,-2)--++(-1,1)--++(2,0);
\draw[middlearrow={>},very thick] (0,0) -- ++(-2,1);
\node at (-.6,.6){$\alpha_1$};
\node at (1.8,.7){$\alpha_2$};
\node at (1.7,-1.7) {$\alpha_3$};
  \node at (1,1)[circle,fill,inner sep=2pt]{};

  \node at (-2,1)[circle,fill,inner sep=2pt]{};

  \node at (1,-1)[circle,fill,inner sep=2pt]{};

  \node at (0,0)[circle,fill,inner sep=2pt]{};

  \node at (2,0)[circle,fill,inner sep=2pt]{};

  \node at (2,-3)[circle,fill,inner sep=2pt]{};

\draw[->, very thick] (3,-1) -- ++(1.5,0);

\draw[middlearrow={>}, very thick] (5.2,0) -- ++(2,0);
\draw[middlearrow={>}, very thick] (7.2,0) -- ++(0,-2);
\draw[middlearrow={>}, very thick] (7.2,-2) -- ++(-2,2);

  \node at (5.2,0)[circle,fill,inner sep=2pt]{};
  \node at (7.2,0)[circle,fill,inner sep=2pt]{};
  \node at (7.2,-2)[circle,fill,inner sep=2pt]{};
  \node at (6.2,.3){$\gamma_1$};
  \node at (7.55,-1){$\gamma_2$};
  \node at (5.9,-1.3){$\gamma_3$};

\end{tikzpicture}
\end{center}
Here we have $\alpha_1 = \gamma_3 - \gamma_1$, $\alpha_2 = -\gamma_1 - \gamma_2$, and $\alpha_3 = -\gamma_2 + \gamma_3$. Thus this fibration corresponds to $([k_1,k_2,k_3],\eta) = ([1,-1,-1],1)$.
\end{ex}

\begin{defn}\label{defn:interiorvertex}
We call a vertex of a $3$-valent signed integer GKM graph an \emph{interior vertex} if the cone spanned by the labels of the three edges emanating from it is equal to $\RR^2$, otherwise it is an \emph{exterior vertex}.
\end{defn}
This notation is motivated by the fact that if we are given a Hamiltonian $T^2$-action with this GKM graph, the momentum image of a fixed point is in the interior of the momentum image if and only if the corresponding vertex of the GKM graph is interior.

\begin{prop}\label{prop:interior}
Let $B$ be a signed $2$-valent GKM graph and $\Gamma\rightarrow B$ the signed GKM fibration associated to $([k_1,\ldots,k_n],\eta)\in((\mathbb{Z}-0)^n/\pm)\times\{0,1\}$ as in Proposition \ref{prop:correspondence}. Then for $i=2,\ldots,n$ the fiber over $v_i$ contains exactly one interior vertex of $\Gamma$ if and only if $k_{i-1}$ and $k_i$ have the same sign. Otherwise both vertices in the fiber are exterior. The fiber over $v_1$ contains exactly one interior vertex if and only if $k_n$ and $(-1)^\eta k_1$ have the same sign. Otherwise both vertices are exterior.
\end{prop}

\begin{proof}
Let $p_i$ and $q_i$ be the vertices in the fiber over $v_i$ for $i\in\{1,\ldots,n\}$. Assume that $\gamma_i$ has been chosen as the weight of the directed edge from $v_i$ to $v_{i+1}$.
Then without loss of generality the set of weights of the edges emanating from $p_i$ and $q_i$ are \[\{-\gamma_{i-1},~\gamma_i,~k_i\gamma_{i-1}-k_{i-1}\gamma_i\}\quad\text{and}\quad\{-\gamma_{i-1},~\gamma_i,~-k_i\gamma_{i-1}+k_{i-1}\gamma_i\}.\]
In general, if $e_1,e_2\in\mathbb{Z}^2$ is a basis then the cone spanned by $e_1$, $e_2$, and $ae_1+be_2$ is $\RR^2$ if and only if $a,b<0$. Thus if $k_{i-1}$ and $k_i$ have the same sign, then exactly one of the set of weights of $p_i$ and $q_i$ has this property. For the statement on $v_1$ recall that in the construction of the fibration we had $k_0=(-1)^{\eta+\varepsilon_1} k_n$. Since $B$ is signed, it follows that $\varepsilon_1=0$, hence $k_0$ and $k_1$ have the same sign if and only if $k_n$ and $(-1)^\eta k_1$ do.
\end{proof}

\begin{cor}\label{cor:nrofintvertices}
Let $\Gamma\rightarrow B$ be a fibration of signed GKM graphs of twisted type as above, where the $2$-valent GKM graph $B$ is effective and has $n$ vertices. If $n$ is odd, then the number of interior vertices of $\Gamma$ is an even number between $0$ and $n-1$. If $n$ is even, then it is an odd number between $1$ and $n-1$.
\end{cor}

\section{Realization of GKM fibrations}\label{sec:geometric realization}

The following is the main theorem of this paper. As before, $T=T^2$ is a two-dimensional compact torus with Lie algebra $\mft$.

\begin{thm} \label{thm:mainthm}
\begin{enumerate}
\item Let $\pi:\Gamma\to B$ be a fiberwise signed GKM fibration, where $\Gamma$ is an $3$-valent integer GKM graph and $B$ an effective $2$-valent integer GKM graph (both with respect to $\mathbb{Z}_\mft^*$). Then $\pi$ is geometrically realized as the projectivization $\PP(E)$ of a $T$-equivariant complex vector bundle $E\to X$ over a four-dimensional $T$-manifold $X$ which can be taken to be $S^4$ if $n=2$ and quasitoric if $n\geq 3$. Furthermore, $\PP(E)$ and $X$ have $T$-invariant stably almost complex structures compatible with the fibration. The realization $\PP(E)$ has the property that all its isotropy groups are connected if and only if, in the notation of Proposition \ref{prop:correspondence} and Remark \ref{rem:correspondencedata}, the fiberwise signed GKM fibration $\pi$ corresponds to $([k_1,\ldots,k_n],\eta)$, with all $k_i=\pm 1$.

\item If $\pi\colon \Gamma\to B$ is a fibration of signed GKM graphs then its geometrical realization as in 1.\ can be chosen to be a fibration of almost complex manifolds such that the induced fibration of signed GKM graphs is precisely $\pi$.

\item If, additionally, $B$ is the boundary of a two-dimensional Delzant polytope (i.e., $X$ can be chosen as a four-dimensional toric manifold), then any realization $\PP(E)$ as in 1.\ admits both a K\"ahler structure and a $T$-invariant symplectic structure such that $\PP(E)\rightarrow X$ is a (holomorphic) symplectomorphism with respect to both structures on $\PP(E)$ and a $T$-invariant K\"ahler structure on $X$. Moreover the underlying symplectic form of the Kähler structure and the invariant symplectic form on $\PP(E)$ are symplectomorphic.
\end{enumerate}
\end{thm}

\begin{rem}\label{rem:mainthm}
In the third part of the above theorem note that the $T$-invariant symplectic form on $\PP(E)$ does admit a compatible complex structure since it is symplectomorphic to a Kähler form. However, this complex structure will not necessarily be compatible with the $T$-action since the K\"ahler form and the symplectomorphism are not. In fact, we will show in Section \ref{sec:nonkaehler} that in the case of a twisted type fibration with the maximal number of interior fixed points, the compatible complex structure can never be $T$-invariant. On the contrary, we will show in Section \ref{sec:product type fibrations} that for product type fibrations we always obtain a $T$-invariant K\"ahler structure on $\PP(E)$. The example of the standard flag manifold $\U(3)/T^3$ shows that such structures can also exist in the twisted type case. However it is not clear whether they always exist outside of the case with maximal number of interior fixed points.
\end{rem}

We will prove the theorem in this section and the following. In Section \ref{sec:geometric realization} we construct the vector bundle $E$, see Theorem \ref{thm:vectorbundle} below. The statements on the geometric structures on $\mathbb{P}(E)$ are proved in Section \ref{S:GeometricStructures}.

\subsection{Realization in dimension 4}

%\begin{lem} Let $T^n$ act effectively on a connected $2n$-dimensional manifold and let $p$ be a fixed point. Then the weights of the isotropy representation at $p$ form a basis of $(\mathbb{Z}^n)^*$.\end{lem}

%\begin{proof} Let $\alpha_1,\ldots,\alpha_n\in (\mathbb{Z^n})^*$ be the weights of the isotropy representation on $T_pM$ (with arbitrarily chosen signs) and assume that they are not a basis. Then $L=\langle \alpha_1,\ldots,\alpha_n\rangle_\mathbb{Z}\subsetneq (\mathbb{Z}^*)^n$ is a proper subset. Consider $L^*:=\{v\in \mathbb{Q}^n~|~\alpha(v)\in \mathbb{Z}\text{ for all }\alpha\in L\}$. We have $\mathbb{Z}^n\subset L^*$, but this is also a proper inclusion as equality would imply $L=(\mathbb{Z}^n)^*$. This means that there is a nontrivial element of the torus on which all the weights vanish. Thus the isotropy representation is not faithful. This implies that the $T^n$-action is ineffective on an equivariant neighbourhood of $p$. But then the principal orbit type is not free and thus the whole action on $X$ is ineffecive.\end{proof}

As a starting point, we need to geometrically realize the base graph of the fibration which corresponds to a $2$-dimensional torus action on a $4$-manifold. Actions of tori of dimension half the dimension of the manifold are quite well-studied so we can draw on the existing theory of quasitoric manifolds.

\begin{prop}\label{prop:dim4realization}
Let $B$ be an effective $2$-regular GKM graph with $n$ vertices. Then:
\begin{enumerate}[(i)]
\item If $n=2$, then $B$ is the GKM graph of a $T^2$-action on $S^4$. For $n\geq 3$, $B$ is the GKM graph of a $4$-dimensional quasitoric manifold $X$.
\item If $B$ has a compatible structure of a signed GKM graph then $X$ can be chosen such that it carries a $T$-invariant almost complex structure.
\item If $B$ is of polytope type, then $X$ can be chosen a toric manifold.
\end{enumerate}
\end{prop}

Part $(iii)$ is Delzant's theorem. In $(i)$ the statement holds for $n=2$ since any such graph can be realized by a $T$-action on $S^4\subset \mathbb{C}^2\oplus \mathbb{R}$ which acts on the $\mathbb{C}^2$ factor as a pullback of the standard representation along some automorphism of $T$; for $n\geq 3$ it follows from the canonical model of quasitoric manifolds which we will now briefly recall in our $4$-dimensional setting (see \cite{MR3363157} for an extensive treatment also in higher dimensions) and translate to GKM graphs. As $B$ is an $n$-gon with $n\geq 3$, the underlying graph is realized as the boundary of a convex polytope $P\subset\mathbb{R}^2$. For an edge $e$ of $P$ we set $T_e\subset T$ to be $\ker\alpha$, where $\alpha$ is the weight (up to sign) associated to $e$ in $B$, interpreted as a homomorphism $T\to S^1$. We define $X:=P\times T/\sim$ where the equivalence relation $\sim$ is generated by

\[(x,s)\sim (x,t)\text{ if }\begin{cases}x\text{ is a vertex of } P\\
x\in e,~ts^{-1}\in T_e\text{ for some edge $e$ of $P$} \end{cases}\]
The space $X$ carries a natural $T$-action on the second factor and clearly $X_1$ is precisely the $T$-space encoded by $\Gamma$. One can show that $X$ is simply-connected and actually carries the structure of a smooth manifold such that the $T$-action is smooth (see \cite{MR3363157}).

Regarding $(ii)$, the existence of invariant almost complex structures on quasi-toric manifolds is already well understood and it only remains to draw the connection to signed GKM graphs. In fact by \cite{Kustarev}, see also \cite[Theorem 7.3.24]{MR3363157}, the existence of an invariant almost complex structure is equivalent to the existence of what is called a positive omniorientation. An omniorientation is equivalent to the datum of an orientation of $X$ and an explicit parameterization of $T_e$ for every edge $e$, i.e.\ a primitive vector $\lambda_e\in \mathbb{Z}_\mft$   (unique up to sign) such that $\lambda_e$ spans the Lie algebra of $T_e$. Let us also recall the notion of positivity: if $e$ and $e'$ are two edges meeting at a fixed point $p\in X$ then $T_e$ acts in non-trivial fashion on the tangent space at $p$ of the $2$-sphere $S^2_{e'}$ belonging to $e'$. The generator $\lambda_e$ defines an identification $S^1\cong T_e$ and thus gives rise to a $T$-invariant almost complex structure on this subspace. Analogously we obtain an almost complex structure on $T_pS^2_{e}$ using the generator $\lambda_{e'}$. In total we obtain an invariant almost complex structure on $T_pX$. The omniorientation is called positive if for every vertex $p$, the orientation of $T_pM$ induced by the almost complex structure agrees with the chosen orientation for $X$, in which case the almost complex structure on these isolated tangent spaces extends to a $T$-invariant almost complex structure on all of $X$ by \cite{Kustarev}.

In order to connect this concept to the structure of a signed GKM graph we need to relate the $\lambda_e$ and the weights defined by the almost complex structure  on $T_pX$. Consider the dual basis $\lambda_e^*,\lambda_{e'}^*\in \mathbb{Z}_\mft^*$ of $\lambda_e,\lambda_{e'}$. One can show \cite[Proposition 7.3.18]{MR3363157} that the weight of $T_pS^2_{e}$ is $\lambda_{e'}^*$ and the weight of $T_{p}S^2_{e'}$ is $\lambda_{e}^*$.

Now assume $B$ has the structure of a signed GKM graph. Reversing the above correspondence, we define an omniorientation of $X$ such that the weights of the isotropy representations at any fixed point -- with the signs induced by the almost complex structure coming from the omniorientation -- agree with the weights of the edges in $B$ starting at the corresponding vertex. Let $e,e'$ be the edges starting at the vertex of a fixed point $p$. Let $\alpha_{e}$, $\alpha_{e'}$ be the weights of $e$, $e'$ and let  $\alpha_e^*,\alpha_{e'}^*\in \mathbb{Z}_\mft$ be the dual basis. Then we choose $\alpha_{e'}^*$ as parametrization for $T_e$. Let us check that this is well defined: suppose $e$ goes from $p$ to $q$ and $e''$ is the other  oriented edge starting at $q$ with weight $\alpha_{e''}$ in $B$. Then transporting $e'$ along $e$ via a connection on $\Gamma$ we obtain the edge $e''$. Thus $\alpha_{e''}\equiv\alpha_{e'}\mod \alpha_e$. This implies $\alpha_{e''}^*=\alpha_{e'}^*$ so we have a well-defined omniorientation.

It remains to prove that it is positive. This is a consequence of the following property of the weights. Let $p,q,e,e',e''$ be as above and let $\overline{e}$ be the reversed orientation of $e$. Then
\begin{align*}
\det (\alpha_{e''},\alpha_{\overline{e}})=-\det(\alpha_{e''},\alpha_e)=\det(\alpha_{e},\alpha_{e''})=\det(\alpha_e,\alpha_{e'})
\end{align*}
where the last equation is again due to $\alpha_{e''}\equiv\alpha_{e'}\mod \alpha_e$. The weights in the determinants on the left and the right are those defined by the omniorientations at $q$ and $p$, in the order coming from an orientation on the polytope. The claim now follows from Proposition 7.3.21 and the subsequent remark in \cite{MR3363157}.

\subsection{A vector bundle over the one-skeleton}\label{sec:constr-sec}

Having realized the 2-regular graph $B$ by a 4-dimensional manifold $X$, it is now our intermediate goal to construct an equivariant complex vector bundle over $X_1$ such that the set of of $0$- and $1$-dimensional orbits in its projectivization is precisely the $T^2$-space encoded in the graph $\Gamma$. In this section, we will use the theory of cohomogeneity one actions \cite{Mostert}, i.e., actions of compact Lie groups on closed manifolds whose principal orbits have codimension one. In particular, we will use the description of cohomogeneity one $G$-actions with orbit space $[0,1]$ in terms of group diagrams $(G,K^+,K^-,H)$: these are collections of compact Lie groups $H\subset K^\pm\subset G$ such that $K^+/H$ and $K^-/H$ are spheres, see \cite[Theorem 4]{Mostert}, \cite{GGZarei}.

We make use of the notation from Remark \ref{rem:correspondencedata}.
We also return to the notation from the proof of Proposition \ref{prop:correspondence}: choose an orientation for the edge in the fiber over $v_1$ with associated weight $\alpha_1$. Using a compatible connection, this inductively defines orientations on the edges in the fibers over all the $v_i$
such that $\alpha_i\equiv \alpha_{i+1}\mod \gamma_i$. We extend this notation with this property to all $i\in\mathbb{Z}$. Transporting an edge around a lift of the path which goes around the base graph once reverses its orientation if $\Gamma$ is of twisted type and preserves the orientation if it is of product type. Hence we have $\alpha_i=(-1)^{\eta}\alpha_{i+n}$, where $\eta=1$ if $\Gamma$ is of twisted type and $\eta=0$ if it is of product type.
Also recall from the proof of Proposition \ref{prop:correspondence} that there are unique $k_i\in \mathbb{Z}$ such that
\[\alpha_i=k_i\gamma_{i-1}-k_{i-1}\gamma_i.\]

We now come to the construction of the vector bundle: for each invariant $2$-sphere in $X$ corresponding to some edge $e_i$ ($i=1,\ldots,n$) we want to construct a $T$-equivariant $\U(2)$-principal bundle $P_i\to S^2$ such that the projectivization of the associated $\CC^2$-bundle $P_i\times_{\U(2)}\CC^2\to S^2$ with respect to the standard representation of $\U(2)$ on $\CC^2$ has as GKM graph exactly $\pi^{-1}(e_i)$, as depicted in Lemma \ref{lem:projuebers2} below. This manifold $P_i$ being a $T$-equivariant $\U(2)$-principal bundle is the same to construct it as a $(T\times \U(2))$-cohomogeneity one manifold such that the $\U(2)$-subaction is free. For some arbitrary integers $a_i,b_i,c_i$ (which we will specify below), define the homomorphisms

\[A_i^+\colon T\rightarrow \U(2),\quad t\mapsto\begin{pmatrix}
a_i\gamma_{i-1}(t)-b_i\gamma_i(t) & \\ & (a_i-k_i)\gamma_{i-1}(t) + (k_{i-1}-b_i)\gamma_i (t)
\end{pmatrix}\]
and
\[
A_i^-\colon T\rightarrow \U(2),\quad t\mapsto\begin{pmatrix}
c_i\gamma_i(t) - a_i\gamma_{i+1}(t) & \\ & (c_i-k_{i+1})\gamma_i(t) + (k_{i}-a_i)\gamma_{i+1}(t)
\end{pmatrix} \]
where the entries of the matrices are to be interpreted as the homomorphisms $T\rightarrow S^1$ corresponding to the respective elements of $\mathbb{Z}_\mft^*$.

Let $P_i$ be the cohomogeneity-one manifold defined by the group diagram $(G,K^+_i,K^-_i,H_i)$; here,
\begin{align*}
G &= T\times \U(2) \\
K^+_i&=\{(t,A_i^+(t))\mid t\in T\},\\
K^-_i&= \{(t,A_i^-(t))\mid t\in T\},\\
\text{and}\quad H_i&=\{(t,A_i^+(t))\mid t\in \ker \gamma_i\}= \{(t,A_i^-(t))\mid t\in \ker \gamma_i\}.
\end{align*}
The last equation holds since by definition $\gamma_{i-1}=-\gamma_{i+1}\mod \gamma_i$. Consequently $H_i$ is contained both in $K^+_i$ and in $K^-_i$. Furthermore $K_i^\pm/H_i\cong S^1$, so that this really defines a valid group diagram. We fix identifications of the outer (non-principal) orbits with $G/K_i^+$ and $G/K_i^-$.

We observe that $\U(2)\cong \{e\}\times \U(2)$ intersects $K^\pm_i$ trivially, so that the $\U(2)$-subaction on $P_i$ is free; dividing out this subaction we obtain a $T^2$-manifold of cohomogeneity one, whose group diagram is given by the projection of the respective subgroups to the $T$-factor, i.e., $(T, T, T, \ker \gamma_i)$. This manifold is equivariantly diffeomorphic to the $S^2\subset X$ corresponding to the edge $e_i$.

Associated to $P_i\to S^2$, we obtain an associated $T$-equivariant $\CC^2$-bundle
\[
E_i:=P_i\times_{\U(2)}\CC^2\longrightarrow S^2,
\]
where $\U(2)$ acts on $\CC^2$ by the standard representation. Let us compute the $T^2$-representation on the fibers over the fixed points in $S^2$ through the identifications
\[
\CC^2 \cong ((T\times \U(2))/K^\pm_i) \times_{\U(2)} \CC^2,\quad v\mapsto [(e,I_2)K_i^\pm,v].
\]
It is given by pulling back the standard $\U(2)$-representation along $A_i^\pm$ since
\[
t\cdot [(e,I_2)K^\pm_i,v] = [(t,I_2)K^\pm_i,v] = [(e,(A_i^\pm(t))^{-1})K^\pm_i,v] = \left[(e,I_2)K^\pm_i,A_i^\pm(t)v\right].
\]
Next we pass to the projectivized bundle $\PP(E_i)\to S^2$, which is a $T$-equivariant $\CC \PP^1$-bundle over $S^2$.
\begin{lem}\label{lem:projuebers2} The projectivization $\PP(E_i)\rightarrow S^2$ is a  fibration of GKM manifolds. On (unsigned) GKM graphs it is given by

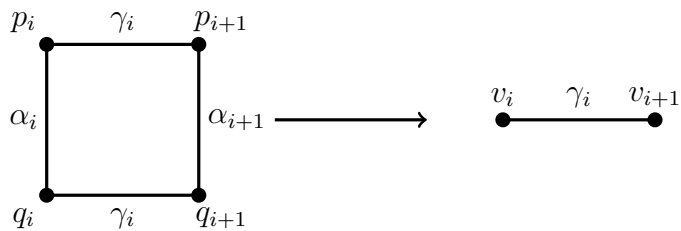
\begin{center}
\begin{tikzpicture}

\draw[->, very thick] (-4,0) -- ++(2,0);

\draw[very thick] (-7,1) -- ++(2,0) -- ++(0,-2) --++(-2,0)--++(0,2);
\node at (-6,1.3){$\gamma_i$};
\node at (-6,-1.3){$\gamma_i$};
\node at (-7.3,0){$\alpha_i$};
\node at (-4.5,0){$\alpha_{i+1}$};

\node at (-7,-1)[circle,fill,inner sep=2pt]{};
\node at (-7.3,-1.3){$q_i$};

\node at (-7,1)[circle,fill,inner sep=2pt]{};
\node at (-7.3,1.3){$p_i$};

\node at (-5,1)[circle,fill,inner sep=2pt]{};
\node at (-4.7,1.3){$p_{i+1}$};

\node at (-5,-1)[circle,fill,inner sep=2pt]{};
\node at (-4.7,-1.3){$q_{i+1}$};

\draw[very thick] (-1,0)--++(2,0);
\node at (0,0.3){$\gamma_i$};

\node at (-1,0)[circle,fill,inner sep=2pt]{};
\node at (-1,0.3){$v_i$};

\node at (1,0)[circle,fill,inner sep=2pt]{};
\node at (1,0.3){$v_{i+1}$};

\end{tikzpicture}
\end{center}
where (with respect to the identifications above) $p_i$ and $p_{i+1}$ correspond to $[1:0]$ while $q_i$ and $q_{i+1}$ correspond to $[0:1]$.
\end{lem}

\begin{proof}
We have a fibration
\[\mathbb{CP}^1\rightarrow \PP(E_i)\rightarrow S^2\] with structure group in $\U(2)$. Thus fixing the standard complex structure on $\mathbb{CP}^1$ and some $T^2$-invariant complex structure on $S^2$, we obtain an almost complex structure on $\PP(E_i)$ such that $\PP(E_i)\rightarrow S^2$ respects almost complex structures. We argue via the corresponding signed GKM graph.

The identifications and orientations are such that the weights at $q_i$ are given by $\pm \gamma_i$ (the sign depending on the chosen almost complex structure on $S^2$) and the upper left entry of $A_i^+$ minus the lower right entry of $A_i^+$ (coming from the fiber over $v_i$). By the definition of $A_i^+$, the latter difference is precisely $k_{i}\gamma_{i-1}-k_{i-1}\gamma_{i}=\alpha_i$. Similar one computes the weights at $q_{i+1}$ to be $\pm\gamma_i$ and $\alpha_{i+1}$. As $\alpha_i=\alpha_{i+1}\mod \gamma_i$, the existence of a connection that is compatible with the signed structure implies that $q_i$ and $q_{i+1}$ are adjacent. Thus in particular the unsigned GKM graph has the form as claimed in the lemma.
\end{proof}

Having constructed the squares in $\Gamma$ over every single edge $e_i$ it remains to glue the $E_i$ in an appropriate and $T$-equivariant manner. We start by gluing $E_{i}$ to $E_{i+1}$ for $i=1,\ldots,n-1$, by identifying the fibers over $v_{i+1}$. From the side of $E_i$ this is the representation of $\CC^2$ defined by the homomorphism $A_i^-$, while from the side of $E_{i+1}$ it is defined by $A_{i+1}^+$. Recall that there are unspecified parameters $a_i,b_i,c_i\in\mathbb{Z}$ in the construction. If we choose them such that $a_{i+1}=c_i$ and $b_{i+1}=a_i$ then $A_i^-=A_{i+1}^+$ and we glue the two fibers in the canonical way.

At this point the one-skeleton of the action on the projectivization is a ladder formed by gluing the individual squares from the lemma above in the obvious way:

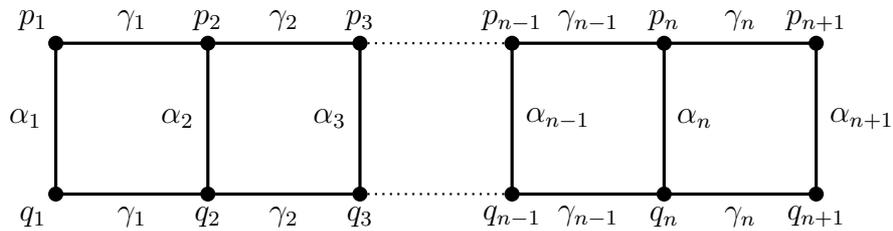
\begin{center}
\begin{tikzpicture}

\draw[very thick] (-7,1) -- ++(4,0) -- ++(0,-2) --++(-4,0)--++(0,2);
\draw[very thick] (-5,1) -- ++(0,-2);
\draw[dotted, thick] (-3,1) -- ++(2,0);
\draw[dotted, thick] (-3,-1) -- ++(2,0);
\draw[very thick] (-1,1) -- ++ (4,0) -- ++(0,-2) --++(-4,0) --++(0,2);
\draw[very thick] (1,1) --++(0,-2);
\node at (-6,1.3){$\gamma_1$};
\node at (-6,-1.3){$\gamma_1$};
\node at (-4,1.3){$\gamma_2$};
\node at (-4,-1.3){$\gamma_2$};
\node at (0,1.3){$\gamma_{n-1}$};
\node at (0,-1.3){$\gamma_{n-1}$};
\node at (2,1.3){$\gamma_{n}$};
\node at (2,-1.3){$\gamma_{n}$};
\node at (-7.4,0){$\alpha_1$};
\node at (-5.4,0){$\alpha_{2}$};
\node at (-3.4,0){$\alpha_{3}$};
\node at (-0.4,0){$\alpha_{n-1}$};
\node at (1.4,0){$\alpha_{n}$};
\node at (3.6,0){$\alpha_{n+1}$};

\node at (-7,-1)[circle,fill,inner sep=2pt]{};
\node at (-7.3,-1.3){$q_1$};

\node at (-7,1)[circle,fill,inner sep=2pt]{};
\node at (-7.3,1.3){$p_1$};

\node at (-5,1)[circle,fill,inner sep=2pt]{};
\node at (-5,1.3){$p_{2}$};

\node at (-5,-1)[circle,fill,inner sep=2pt]{};
\node at (-5,-1.3){$q_{2}$};

\node at (-3,1)[circle,fill,inner sep=2pt]{};
\node at (-3,1.3){$p_{3}$};

\node at (-3,-1)[circle,fill,inner sep=2pt]{};
\node at (-3,-1.3){$q_{3}$};

\node at (-1,1)[circle,fill,inner sep=2pt]{};
\node at (-1,1.3){$p_{n-1}$};

\node at (-1,-1)[circle,fill,inner sep=2pt]{};
\node at (-1,-1.3){$q_{n-1}$};
\node at (1,1)[circle,fill,inner sep=2pt]{};
\node at (1,1.3){$p_{n}$};

\node at (1,-1)[circle,fill,inner sep=2pt]{};
\node at (1,-1.3){$q_{n}$};
\node at (3,1)[circle,fill,inner sep=2pt]{};
\node at (3,1.3){$p_{n+1}$};

\node at (3,-1)[circle,fill,inner sep=2pt]{};
\node at (3,-1.3){$q_{n+1}$};

\end{tikzpicture}
\end{center}
In order to obtain the graph $\Gamma$ it remains to glue $E_{n+1}$ to $E_1$. This needs to be done with respect to the graph theoretical structure of $\Gamma$. If $\Gamma$ is of product type, we obtain $\Gamma$ from the ladder by identifying $p_{n+1}=p_1$ as well as $q_{n+1}=q_1$. In the twisted case we need to identify $p_{n+1}=q_1$ and $q_{n+1}=p_1$.

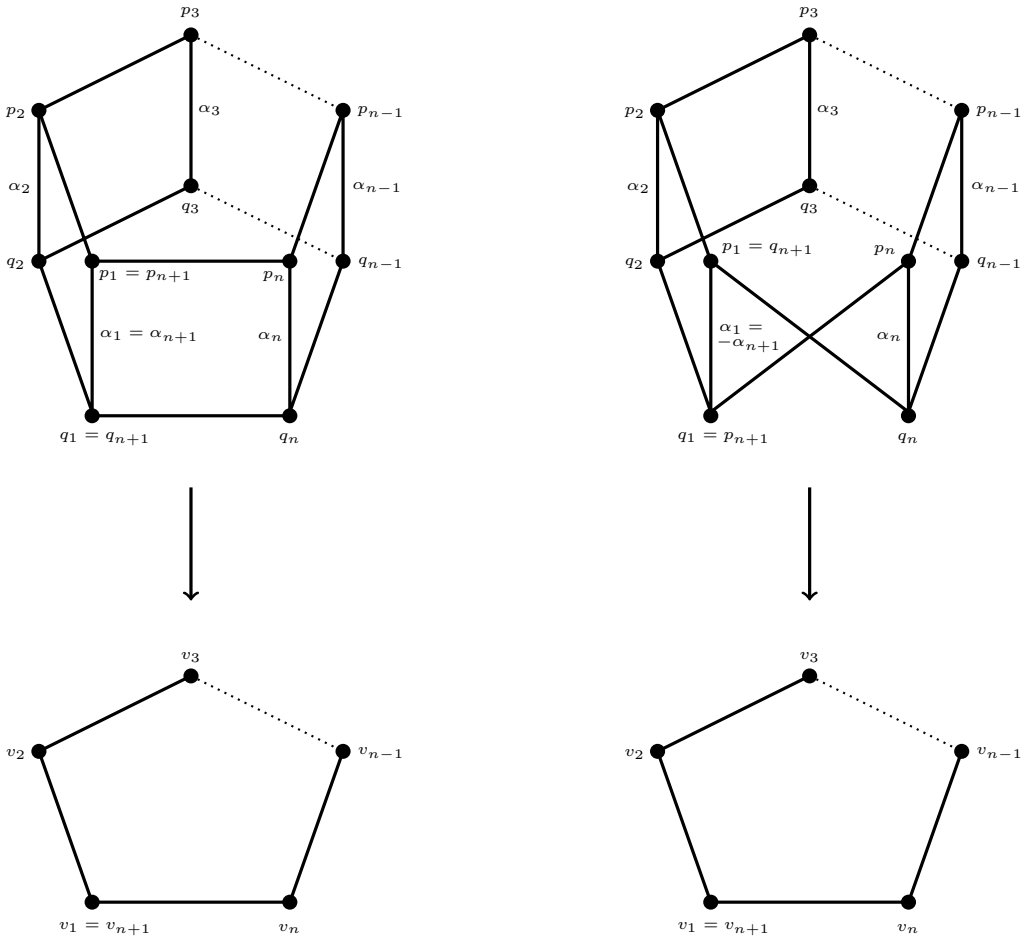
\begin{center}
\begin{tikzpicture}

\draw[very thick] (-7.3,1) -- ++(-.7,2) -- ++(2,1) --++(0,-2)--++(-2,-1) --++(.7,-2) --++(0,2);
\draw[very thick] (-8,3) -- ++(0,-2);
\draw[dotted, thick] (-6,4) -- ++(2,-1);
\draw[dotted, thick] (-6,2) -- ++(2,-1);
\draw[very thick] (-4,3) -- ++ (-.7,-2) -- ++(0,-2) --++(.7,2) --++(0,2);
\draw[very thick] (-4.7,1) -- (-7.3,1);
\draw[very thick] (-4.7,-1.05) -- (-7.3,-1.05);

\node at (-7.3,1)[circle,fill,inner sep=2pt]{};
\node at (-7.3,-1.05)[circle,fill,inner sep=2pt]{};
\node at (-8,3)[circle,fill,inner sep=2pt]{};
\node at (-6,4)[circle,fill,inner sep=2pt]{};
\node at (-6,2)[circle,fill,inner sep=2pt]{};
\node at (-8,1)[circle,fill,inner sep=2pt]{};
\node at (-4,3)[circle,fill,inner sep=2pt]{};
\node at (-4,1)[circle,fill,inner sep=2pt]{};
\node at (-4.7,1)[circle,fill,inner sep=2pt]{};
\node at (-4.7,-1.05)[circle,fill,inner sep=2pt]{};

\node at (-6.6, .8){\tiny $p_1 = p_{n+1}$};
\node at (-8.3,2.97){\tiny $p_2$};
\node at (-6,4.3) {\tiny $p_3$};
\node at (-3.5,2.97) {\tiny $p_{n-1}$};
\node at (-4.9,.8) {\tiny $p_n$};

\node at (-7.13,-1.35){\tiny $q_1 = q_{n+1}$};
\node at (-8.3,.97) {\tiny $q_2$};
\node at (-6,1.7) {\tiny $q_3$};
\node at (-3.5,.97) {\tiny $q_{n-1}$};
\node at (-4.7,-1.35){\tiny $q_n$};

\node at (-6.55,0) {\tiny $\alpha_1 = \alpha_{n+1}$};
\node at (-8.25,1.97) {\tiny $\alpha_2$};
\node at (-5.75,3) {\tiny $\alpha_3$};
\node at (-3.55,1.97) {\tiny $\alpha_{n-1}$};
\node at (-4.95,0) {\tiny $\alpha_n$};

\draw[->, very thick] (-6,-2) -- ++(0,-1.5);

\draw[very thick] (-7.3,-7.5) -- ++(-.7,2) -- ++(2,1) ;
\draw[dotted, thick] (-6,-4.5) -- ++(2,-1);
\draw[very thick] (-4,-5.5) -- ++(-.7,-2) --++(-2.6,0);

\node at (-7.3,-7.5)[circle,fill,inner sep=2pt]{};
\node at (-8,-5.5)[circle,fill,inner sep=2pt]{};
\node at (-6,-4.5)[circle,fill,inner sep=2pt]{};
\node at (-4,-5.5)[circle,fill,inner sep=2pt]{};
\node at (-4.7,-7.5)[circle,fill,inner sep=2pt]{};

\node at (-7.13,-7.85){\tiny $v_1 = v_{n+1}$};
\node at (-8.3,-5.53) {\tiny $v_2$};
\node at (-6,-4.25) {\tiny $v_3$};
\node at (-3.5,-5.53) {\tiny $v_{n-1}$};
\node at (-4.7,-7.85){\tiny $v_n$};

\end{tikzpicture}
\qquad \qquad \qquad
\begin{tikzpicture}

\draw[very thick] (-7.3,1) -- ++(-.7,2) -- ++(2,1) --++(0,-2)--++(-2,-1) --++(.7,-2) --++(0,2);
\draw[very thick] (-8,3) -- ++(0,-2);
\draw[dotted, thick] (-6,4) -- ++(2,-1);
\draw[dotted, thick] (-6,2) -- ++(2,-1);
\draw[very thick] (-4,3) -- ++ (-.7,-2) -- ++(0,-2) --++(.7,2) --++(0,2);
\draw[very thick] (-4.7,1) -- (-7.3,-1);
\draw[very thick] (-4.7,-1) -- (-7.3,1);

\node at (-7.3,1)[circle,fill,inner sep=2pt]{};
\node at (-7.3,-1.05)[circle,fill,inner sep=2pt]{};
\node at (-8,3)[circle,fill,inner sep=2pt]{};
\node at (-6,4)[circle,fill,inner sep=2pt]{};
\node at (-6,2)[circle,fill,inner sep=2pt]{};
\node at (-8,1)[circle,fill,inner sep=2pt]{};
\node at (-4,3)[circle,fill,inner sep=2pt]{};
\node at (-4,1)[circle,fill,inner sep=2pt]{};
\node at (-4.7,1)[circle,fill,inner sep=2pt]{};
\node at (-4.7,-1.05)[circle,fill,inner sep=2pt]{};

\node at (-6.55, 1.15){\tiny $p_1 = q_{n+1}$};
\node at (-8.3,2.97){\tiny $p_2$};
\node at (-6,4.3) {\tiny $p_3$};
\node at (-3.5,2.97) {\tiny $p_{n-1}$};
\node at (-5,1.15) {\tiny $p_n$};

\node at (-7.13,-1.35){\tiny $q_1 = p_{n+1}$};
\node at (-8.3,.97) {\tiny $q_2$};
\node at (-6,1.7) {\tiny $q_3$};
\node at (-3.5,.97) {\tiny $q_{n-1}$};
\node at (-4.7,-1.35){\tiny $q_n$};

\node at (-6.9,.1) {\tiny $\alpha_1=$};
\node at (-6.8,-.1){\tiny $ -\alpha_{n+1}$};

\node at (-8.25,1.97) {\tiny $\alpha_2$};
\node at (-5.75,3) {\tiny $\alpha_3$};
\node at (-3.55,1.97) {\tiny $\alpha_{n-1}$};
\node at (-4.95,0) {\tiny $\alpha_n$};

\draw[->, very thick] (-6,-2) -- ++(0,-1.5);

\draw[very thick] (-7.3,-7.5) -- ++(-.7,2) -- ++(2,1) ;
\draw[dotted, thick] (-6,-4.5) -- ++(2,-1);
\draw[very thick] (-4,-5.5) -- ++(-.7,-2) --++(-2.6,0);

\node at (-7.3,-7.5)[circle,fill,inner sep=2pt]{};
\node at (-8,-5.5)[circle,fill,inner sep=2pt]{};
\node at (-6,-4.5)[circle,fill,inner sep=2pt]{};
\node at (-4,-5.5)[circle,fill,inner sep=2pt]{};
\node at (-4.7,-7.5)[circle,fill,inner sep=2pt]{};

\node at (-7.13,-7.85){\tiny $v_1 = v_{n+1}$};
\node at (-8.3,-5.53) {\tiny $v_2$};
\node at (-6,-4.25) {\tiny $v_3$};
\node at (-3.5,-5.53) {\tiny $v_{n-1}$};
\node at (-4.7,-7.85){\tiny $v_n$};

\end{tikzpicture}
\end{center}

In any case we wish to identify the representations defined by $A_{n+1}^-$ and $A_1^+$. Recall that $\gamma_0=(-1)^{\varepsilon_0}\gamma_n$, $\gamma_1=(-1)^{\varepsilon_1}\gamma_{n+1}$, and $\alpha_1=(-1)^\eta\alpha_{n+1}$ for $\varepsilon_i,\eta\in\{0,1\}$. We have

\[
k_1\gamma_0-k_0\gamma_1=(-1)^\eta(k_{n+1}\gamma_n-k_n\gamma_{n+1})=(-1)^{\eta+\varepsilon_0}k_{n+1}\gamma_0-(-1)^{\eta+\varepsilon_1}k_n\gamma_1
\]
which implies $k_0=(-1)^{\eta+\varepsilon_1}k_n$ and $k_1=(-1)^{\eta+\varepsilon_0} k_{n+1}$. Consequently the matrices defining $A_{n+1}^-$ and $A_1^+$ can be written as

\[\begin{pmatrix}
(-1)^{\varepsilon_0} c_n\gamma_0 - (-1)^{\varepsilon_1} a_n\gamma_{1} & \\ & ((-1)^{\varepsilon_0}c_n-(-1)^{\eta}k_{1})\gamma_0 + ((-1)^{\eta}k_{0}-(-1)^{\varepsilon_1}a_n)\gamma_{1}
\end{pmatrix},\]
and

\[\begin{pmatrix}
a_1\gamma_{0}-b_1\gamma_1 & \\ & (a_1-k_1)\gamma_{0} + (k_{0}-b_1)\gamma_1
\end{pmatrix}.\]

Thus in the product case, if we have $a_1=(-1)^{\varepsilon_0}c_n$ and $b_1=(-1)^{\varepsilon_1} a_n$ then $A^-_{n+1}=A^+_1$ and the two naturally glue, finishing the construction. Observe that an arbitrary choice of either the $a_i$, the $b_i$, or the $c_i$ for $i=1,\ldots,n$ uniquely defines the respective other coefficients.

In the twisted case we set $c_n=(-1)^{\varepsilon_0}(a_1-k_1)$ and $a_n=(-1)^{\varepsilon_1}(b_1-k_0)$. Again, an arbitrary choice of the $a_i$, the $b_i$, or the $c_i$ defines the other coefficients uniquely such that this and the previous gluing conditions hold. Now $A_{n+1}^-$ and $A_1^+$ do not agree but arise from one another by swapping the diagonal entries. It follows that the automorphism of $\mathbb{C}^2$ which swaps both factors is equivariant with respect to the actions defined by $A_{n+1}^-$ and $A_1^+$. We use this automorphism to glue $E_n$ to $E_1$ along the fiber over $v_1$. Since the induced automorphism of $\mathbb{CP}^1$ swaps the fixed points $[0:1]$ and $[1:0]$ it follows that in the graph encoding the one-skeleton of the projectivization, $p_{n+1}$ gets glued to $q_1$ and $q_{n+1}$ gets glued to $p_1$ and is thus precisely $\Gamma$.

%{{{ Section:Obstruction theory for quasitoric manifolds
\subsection{Obstruction theory for quasitoric manifolds}\label{S:ObstructionTheory}
Let $X$ be a quasitoric manifold of dimension $4k$. Denote by $\pi \colon X \to P$ the projection
to the associated simple $2k$-polytope $P$. The preimage under $\pi$ of the interior of $P$ is diffeomorphic
to $T^{2k} \times \mathring D^{2k}$ where $\mathring D^{2k}$ denotes the interior of the unit
closed ball in $\RR^{2k}$. Furthermore this space can be regarded as the interior of the equivariant free top cell of
$X$: recall from the canonical model of a quasitoric manifold (see the proof of Proposition \ref{prop:dim4realization}) that $X$ is equivariantly homeomorphic to a quotient $ T^{2k}\times P/\sim$. Identifying $P$ with $D^{2k}$, the natural projection $T^{2k}\times D^{2k}\rightarrow T^{2k}\times P/\sim$ becomes a characteristic map for a relative CW structure on $(X,A)$, where $A$ is the preimage of the boundary of $P$ under $\pi$.

Note that, although it is not quasitoric, we have the same kind of cellular structure for the $T^2$-action on $X=S^4\subset \mathbb{C}^2\oplus \mathbb{R}$ which acts in standard fashion on the $\mathbb{C}^2$ factor. The space $X$ arises from $A=X_1=\{(v,w,z)\in S^4~|~v=0\text{ or }w=0\}$ by attaching a single free cell.
To see this consider the space $D=\{(x,y,z)\in S^3~|~x,y\geq 0\}$, which is a $2$-disk. The map $D\times T^2\rightarrow S^4$, defined as the equivariant extension of the map $D\times\{e\}\cong D\rightarrow S^4$ induced by the inclusion $\mathbb{R}\rightarrow\mathbb{C}$ on the first two components, is a characteristic map for the relative CW structure. Of course this CW decomposition induces an analogous decomposition for any pullback of the above action along a group automorphism.

\begin{lem}\label{L:extension of equivariant vector bundles}
   Suppose $(X,A)$ is the equivariant relative $T^{2k}$-CW complex defined above and
  $E \to A$ is an equivariant complex vector bundle of rank $r$ with $r>k-1$. Then $E$ can
  be extended to an equivariant vector bundle over $X$.
\end{lem}
\begin{proof}
We will use equivariant obstruction theory to prove this lemma, cf. \cite[Chapter 2, Section
  3]{MR889050}. Since $E \to A$ is a $T^{2k}$-equivariant complex vector bundle, we have a map
    $\varepsilon \colon A \to B(T^{2k},U(r))$, where $B(T^{2k},U(r))$ is the classifying space of
    $T^{2k}$-equivariant $U(r)$-principal bundles, see \cite[Section 3.1]{MR0245027}. We would like
    to extend $\varepsilon$ over the top cell $e^{T^{2k}}_{2k}$ of $X$.

   Now $\varepsilon$ has an equivariant extension
  to $X = A \cup_{\varphi} e^T_{2k}$ if the composition of maps
  \[
    S^{2k-1} \longrightarrow T^{2k} \times S^{2k-1} \overset{\varphi}{\longrightarrow} A \overset{\varepsilon}{\longrightarrow}
    B(T^{2k},U(r))
  \]
  is nullhomotopic, see \cite[pp. 115]{MR889050}. The composition gives
  an element of $\pi_{2k-1}( B(T^{2k},U(r) ) )$, but the classifying space
  $B(T^{2k},U(r))$ is homotopy equivalent to $BU(r)$ \cite[p.\ 142]{MR0245027} and since $r > k-1$
  we have by Bott periodicity $\pi_{2k-1}(B(T^{2k},U(r)))=0$. This implies that $\varepsilon$ can be
  extended equivariantly to a map $\varepsilon \colon X \to B(T^{2k},U(r))$.

\end{proof}
% }}}

\begin{rem}\label{R:Equivariant vector bundles is smooth as fuck}
As $X$ is smooth, the complex bundle $E\rightarrow X$ in the previous lemma can be constructed in a smooth fashion as well. To see this, note that $E$ admits an equivariant, fiberwise injective and linear map $f\colon E\rightarrow W$ to some finite-dimensional complex $T^{2k}$-representation (as in the proof of \cite[Chapter I, Proposition 9.7]{MR889050}). Denote by $G_r(W)$ the Grassmanian of complex $r$-planes (where $r$ is the rank of $E$) and by $E_r\rightarrow G_r(W)$ the tautological bundle.  The action on $W$ extends to an action on $G_r(W)$ and $E_r$, making the latter an equivariant vector bundle. The map $f$ induces a bundle map $g\colon E\rightarrow E_r$  which maps each fiber $E_x$ isomorphically to the fiber over $f(E_x)\in G_r(W)$ via $f$. Thus the underlying equivariant map $\bar{g}\colon X\rightarrow G_r(W)$ has the property that $\bar{g}^*(E_r)\cong E$. By \cite[Corollary 1.12]{MR0250324} we may equivariantly homotope $\bar{g}$ to a smooth map. The pullback along the latter is naturally smooth and still equivariantly equivalent to $E$ (see \cite[Theorem 8.15]{MR889050}).
\end{rem}

Applying Lemma \ref{L:extension of equivariant vector bundles} to the vector bundle
constructed in Section \ref{sec:constr-sec} we arrive at the following theorem, which is an intermediate step in the proof of Theorem \ref{thm:mainthm}.

\begin{thm}\label{thm:vectorbundle}
Let $\Gamma\rightarrow B$ be a GKM fibration as in Theorem \ref{thm:mainthm} and $X$ be the realization of $B$ from Proposition \ref{prop:dim4realization}. We fix the notation from Remark \ref{rem:correspondencedata}.
For any $(a_1,\ldots,a_n)\in \mathbb{Z}^n$ there exists a smooth $T^2$-equivariant complex vector bundle $E\rightarrow X$ of rank $2$ with the following properties:
\begin{enumerate}[(i)]
\item the isotropy representation of $E$ over $v_i$, $i=1,\ldots,n$, is isomorphic to the pullback of the standard action on $\mathbb{C}^2$ along the homomorphism $T\rightarrow \U(2)$,
\[t\mapsto\begin{pmatrix}
a_i\gamma_{i-1}(t)-a_{i-1}\gamma_i(t) & \\ & (a_i-k_i)\gamma_{i-1}(t) + (-a_{i-1}+k_{i-1})\gamma_i (t)
\end{pmatrix},\]
where $([k_1,\ldots,k_n],\eta)$ is defined by the fibration in the sense of correspondence \ref{prop:correspondence} and we set $a_0=(-1)^{\varepsilon_1}a_n+\eta k_0$ and $k_0=(-1)^{\varepsilon_1+\eta}k_n$.
\item The fibration $\PP(E)\rightarrow X$ is one of GKM manifolds and realizes the GKM fibration $\Gamma\rightarrow B$.
\end{enumerate}
\end{thm}
We conclude this section with the proof of the statement on the isotropy groups of the $T$-action on $\PP(E)$ in part 1.\ of Theorem \ref{thm:mainthm}. Recall first that the weights of the isotropy representation at any of the two fixed points over a vertex $v_i$ of $B$ are (up to sign) $\gamma_{i-1}$, $\gamma_i$ and $\alpha_i = k_i\gamma_{i-1} - k_{i-1}\gamma_i$. If any $k_j\neq \pm 1$, we can therefore find a fixed point in $\PP(E)$ such that two of the adjacent weights, considered as characters $T\to S^1$, have a nontrivial common kernel, which then occurs as an isotropy group near the fixed point. Conversely, we assume that all $k_i=\pm 1$. As $\PP(E)$ fibers equivariantly over the $T$-manifold $X$, and in $X$ nontrivial stabilizers occur only in the one-skeleton $X_1$, any nontrivial stabilizer in $\PP(E)$ necessarily occurs in $\pi^{-1}(X_1)$. This space however is, in the notation of Section \ref{sec:constr-sec}, the union of four-dimensional $T^2$-manifolds $\PP(E_i)$ that fiber over $S^2$. By assumption on the $k_i$ the $T$-action on $\PP(E_i)$ is effective. Hence, the $\PP(E_i)$ are torus manifolds with vanishing odd-degree cohomology, which by \cite[Theorem 4.1]{MasudaPanov} have only connected isotropy groups. See Remark \ref{rem:connectedstabs} for an example of a linear realization of a total space of a GKM fibration where not all $k_i=\pm 1$.

\section{Geometric structures on the realization}\label{S:GeometricStructures}

\subsection{(Stably) almost complex structures}\label{sec: sacs}

The realization
\[\mathbb{CP}^1\longrightarrow \PP(E)\longrightarrow X\]
constructed in the previous sections is a fibration with structure group $\U
(2)$. Let $V_F\subset T\PP(E)$ be the subbundle consisting of the tangent spaces of all the fibers. Since each fiber can be identified with $\mathbb{CP}^1$ uniquely up to elements from $\U(2)$, we see that a $\U(2)$-invariant almost complex structure on $\mathbb{CP}^1$ induces on $V_F$ the structure of a complex $T$-vector bundle. A complement of $V_F$ in $\PP(E)$ with respect to a $T$-invariant metric can be identified with the pullback $V_X$ of the tangent bundle of $X$ along the projection $\PP(E)\rightarrow X$. As $X$ is quasitoric, it is $T$-invariantly stably almost complex (see \cite[Corollary 7.3.15]{MR3363157}). It follows that $T\PP(E)=V_F\oplus V_X$ carries a $T$-invariant stable almost complex structure. This finishes the proof of part 1 of Theorem \ref{thm:mainthm}.

In case the graph fibration is one of signed graphs, $X$ carries a $T$-invariant almost complex structure by Proposition \ref{prop:dim4realization}. As above it follows that $\PP(E)$ carries an invariant almost complex structure compatible with the map to $X$. For the proof of part 2 of Theorem \ref{thm:mainthm}, it remains to check that the signed GKM graphs of $\PP(E)\rightarrow X$ with respect to the almost complex structures agree with the given fibration of signed graphs $\Gamma\rightarrow B$. The construction of $X$ and its almost complex structure (see Proposition \ref{prop:dim4realization}) can be carried out such that its signed GKM graph is precisely $B$. Since $\PP(E)\rightarrow X$ is compatible with the complex structures, it follows that the given labelling function on oriented edges of ${\Gamma}$ agrees with that induced by the almost complex structure on $\PP(E)$ when applied to oriented basic edges. On (oriented) fiber edges they a priori only agree up to sign. However we observe that, since both labelling functions on oriented edges admit a compatible connection, the signs in the fiber edges either agree everywhere or nowhere. In the first case we are done. In the second case we may apply the automorphism of $\Gamma$ that interchanges the vertices in each fiber. This transforms one labelling function into the other so the proof of part 2 is complete.

\subsection{Symplectic and Kähler structures} \label{sec:sympKaehler}

Suppose $E \to X$ is a complex vector bundle of rank $2$ and $X$ a toric $4$-manifold. First note that
$X$ is a smooth projective variety \cite[Proposition 5.2.2]{MR3363157} and thus $X$ admits in
particular a Kähler structure. We denote by $\omega_X$ the corresponding Kähler form. Since $X$ is
projective we deduce from \cite[Theorem 9]{MR0137711} that $E$ is algebraic if the determinant bundle
$\det E$ is algebraic. But since $X$ is toric, $H^2(X) = H^{1,1}(X)$, so that $\det E$ has a holomorphic structure by the Lefschetz theorem on $(1,1)$-classes, cf. \cite{MR0033557}. With the GAGA principle
\cite{MR0082175} we infer that $\det E$ is algebraic, hence $E$. %Finally $E \to X$ has a holomorphic structure, again by the GAGA principle.
Since $E$ is algebraic, the manifold
$\mathbb P(E)$ is a smooth projective variety. %\marginpar{\tiny Genauer erklären wieso das Ding projektiv ist?}

From the discussion above we obtain

%{{{ Proposition:algebraic Bundles
\begin{prop}\label{P: projectivization is complex manifold}
For every rank $2$ complex vector bundle $E \to X$ the projectivization $\mathbb P(E)$
  is a smooth projective variety such that $\pi \colon \mathbb P(E) \to X$ is a holomorphic map.
\end{prop}

% }}}

In the remaining part of this section we will construct the symplectic and K\"ahler structure from Theorem \ref{thm:mainthm}.

Now let $E\rightarrow X$ be a complex, equivariant $T$-bundle. Furthermore let $h$ be $T$-invariant Hermitian metric on $E$ and let $\omega_X$ be the $T$ invariant K\"ahler form on $X$.
Let $L \to \mathbb P(E)$ be the tautological bundle (i.e.\ $L$
 restricted to each fiber of $\mathbb P(E)$ is the tautological bundle of $\CC\PP^{1}$).
Since by construction $L$ is a subbundle of $\pi^\ast(E)$, the metric $h$ can
be restricted to $L$ and it also induces a metric $h^*$ on the dual bundle $\overline L$. The Chern form $\omega_V$ of the Hermitian line bundle $(\overline L,h^*)$ has the property that it restricts to the Fubini-Study form (associated to the metric $h$) on every fiber of $\mathbb{P}(E) \to X$ and that furthermore
\[
  \omega_{K} := \omega_V + C \cdot \pi^{\ast}(\omega_X)
\]
is a Kähler form on $\mathbb{P}(E)$ for sufficiently large $C$ (see \cite[Section 3.3.2]{Voisin}).

The pullback $\pi^{\ast}(\omega_X)$ is $T$-invariant but $\omega_V$ and hence also $\omega_K$ in general are not. However since $T$ is compact, we can average $\omega_V$ over $T$, obtaining another form $\omega_F$ which
is $T$-invariant. As $h$ is $T$-invariant, the $T$ action preserves the Fubini-Study form on the fibers, so for any fiber of $\PP(E)\rightarrow X$ and $t\in T$ we have $i^*(t^*\omega_V)=i^*\omega_V$, where $i$ is the fiber inclusion. This implies that also the averaged form $\omega_F$ restricts to the Fubini-Study form on every fiber. The latter is in particular symplectic, thus after possibly adjusting the constant $C$ the form
\[
  \omega_S := \omega_F + C \cdot \pi^{\ast}(\omega_X)
\]
is a $T$-invariant symplectic form on $\mathbb{P}(E)$. Since $H^1(\mathbb{P}(E);\ZZ)=0$ we infer
from \cite[Addendum to Theorem 28.1]{MR770935} that the action is in fact Hamiltonian.

\begin{prop}\label{prop: symplectomorphic forms}
The Kähler form $\omega_K$ and the invariant symplectic form $\omega_S$ are symplectormorphic.
\end{prop}

\begin{proof}
Since averaging a closed form over a compact group does not change its deRham class there is
  a $1$-form $\eta$ such that $\omega_F = \omega_V + d\eta$. Because $\omega_V$ and $\omega_F$
  agree on the fibers of $\mathbb{P}(E) \to X$ it follows that $d\eta$ restricts to $0$. In particular, for any $t \in \RR$ the forms $\omega_V
  +td\eta$ all restrict to symplectic forms on the fibers. Therefore the forms
  $\omega_t := \omega_V +t d\eta + C \cdot \pi^*(\omega_X)$ are symplectic for $t \in [0,1]$ with
  an eventually bigger constant $C$. Thus $\omega_K$ and $\omega_S$ are joined by a path of
  symplectic forms in the same deRham class and by Moser's trick they are symplectomorphic.
\end{proof}

This completes the proof of Theorem \ref{thm:mainthm}. In the next section we will show that GKM fibrations of product type always admits a toric symplectic realization, as announced in Remark \ref{rem:mainthm}.

\subsection{Fibrations of product type: the toric case}\label{sec:product type fibrations}
In this section we consider a signed GKM fibration $\pi:\Gamma\to B$ of product type, where $B$ is the boundary of a two-dimensional Delzant polytope $P\subset \mft^*$. Let $X$ be the four-dimensional toric manifold with Delzant polytope $P$.

In this setting, we can refine the construction in Section \ref{sec:constr-sec}. The $T$-equivariant complex rank $2$ vector bundle $E\to X_1$ over the one-skeleton of $X$ constructed there is, because the graph $\Gamma$ is of product type, the Whitney sum of two $T$-equivariant complex line bundles $L_i\to X_1$, $i=1,2$.  By Lemma \ref{L:extension of equivariant vector bundles} we can extend not only the bundle $E$ equivariantly to all of $X$, but also the bundles $L_i$. In this way, we obtain an extension  $E\to X$ which is globally the sum of two equivariant line bundles.

Now, any complex line bundle admits a canonical circle action, so that $E$ admits a canonical auxiliary action of a two-dimensional torus, which commutes with the $T$-action. Note that the diagonal action becomes trivial after passing to the projectivization $\PP(E)$, but the fiberwise circle action given in homogeneous coordinates by $t\cdot [z:w]=[tz:w]$ defines an effective circle action on $\PP(E)$ commuting with the $T$-action. In total we obtain a $T^3$-action on $\PP(E)$.

As explained in Section \ref{sec:sympKaehler}, the projectivization $\PP(E)$ admits a $T$-invariant symplectic structure. We can modify this construction to make the symplectic form invariant also under the $T^3$-action. We thus see that $\PP(E)$ admits the structure of a toric symplectic manifold. Note that toric structures on projectivizations of sums of circle bundles have been considered before, see e.g.\ \cite{ChoiPark}.

\begin{rem}
One can construct a $6$-dimensional toric manifold such that the restriction of the action to a two-dimensional subtorus has the prescribed GKM graph $\Gamma$ via Delzant's theorem:

Let $T^3 = T^2\times S^1$, with Lie algebra $\mft^3 = \mft\oplus \RR$ with integral lattice $\ZZ^2\times \ZZ$, and dual $(\mft^3)^* \cong \mft^* \oplus \RR \cong \RR^3$. We construct a Delzant polytope $Q\subset (\mft^3)^*$ as follows: in the hyperplane $z=0$ it contains the polytope $P\subset \mft^*\cong \mft^*\oplus \{0\}\subset (\mft^3)^*$. As in Section \ref{sec:constr-sec}, we let $v_1,\ldots,v_n$ be the vertices of $P$, and $\gamma_i$ the label of the edge $e_i$ from $v_i$ to $v_{i+1}$. The only other vertices of $Q$ are the elements $w_i:=v_i + (\alpha_i,1)$, all contained in the hyperplane $z=1$.

Then $Q$ is a Delzant polytope: The slopes of the edges emerging from the vertex $v_i$ are $-\gamma_{i-1}, \gamma_i$, and $(\alpha_i,1)$. These three vectors form a basis of $\ZZ^3$ because the $\gamma_{i-1}$ and $\gamma_i$ form a basis of $\ZZ^2$. The edge from $w_i$ to $w_{i+1}$ points in direction $v_{i+1}-v_{i} + \alpha_{i+1} - \alpha_{i}$, which is a multiple of $\gamma_i$; this shows that the Delzant condition is also satisfied at the vertices $w_i$.

By construction, the six-dimensional toric symplectic manifold with Delzant polytope $Q$ satisfies that the restriction of the $T^3$-action to $T^2\cong T^2\times \{1\}$ is GKM, with (signed) GKM graph $\Gamma$.
\end{rem}

\section{Non-K\"ahler GKM graphs}
\label{sec:nonkaehler}

The goal of this section is to show that certain GKM actions do not admit compatible K\"ahler structures. The necessary criteria to show that signed GKM structures can not come from such an action were already described in Section \ref{sec:defsec}. However, in order to show the non-existence of such structures for a given action, there is more work to do. The results of this section are collected in the following

\begin{thm}\label{thm:signed-classification}
Let $\Gamma\rightarrow B$ be a signed fibration of GKM graphs of twisted type such that $\Gamma$ is $3$-valent and $B$ is the boundary of a 2-dimensional Delzant polytope with $n$ vertices (labels of the edges are with respect to a rank $2$ lattice $\mathbb{Z}_\mft^*$). We assume that $\Gamma$ has $n-1$ interior vertices (cf.\ Definition \ref{defn:interiorvertex} and Corollary \ref{cor:nrofintvertices}). Then, up to isomorphism, there are up to three structures of signed GKM graphs which are compatible with the underlying unsigned GKM graph of $\Gamma$. They have the following properties:

\begin{enumerate}[I)]
\item the signed graph $\Gamma$ itself. This signed structure can not be realized by a complex structure compatible with an invariant Kähler form.
\item the signed structure which arises from $\Gamma$ by changing the sign of the weights of every second pair of basic edges. It exists only if $n$ is even. If furthermore $n\neq 4$ then this signed structure is not realized by a Hamiltonian action.
\item a unique signed structure where no two basic edges over the same edge have the same weights. It exists only if $k_i=\pm 1$ for $i=1,\ldots,n$ (as in Section \ref{sec:graphs in dim 6}) and is not realized by a Hamiltonian action.
\end{enumerate}
\end{thm}

\begin{cor}
A GKM action whose unsigned GKM graph is that of $\Gamma$ as above with $n\neq 4$ does not admit an invariant K\"ahler structure.
\end{cor}

\begin{rem}
We have seen that case I is always realized as a Hamiltonian action. For
even $n$, case II is realizable as an equivariant fibration of almost complex manifolds since it is a signed fibration
over the signed base graph which arises from $B$ by changing the sign of
every second edge. Note that by the discussion in Section
\ref{S:GeometricStructures} we may use the same underlying $T^2$-manifold
for both cases and only need to vary the almost complex structures. We do
not know about the realizability of case III.
\end{rem}

\subsection{Case I}\label{subsec:case I}

We show that the signed structure $\Gamma$ is indeed not realized by an invariant K\"ahler form.

\begin{lem}\label{lem:connectinteriorvertex}
If two fibers in $\Gamma$ are connected by a basic edge and both contain an interior vertex, then the interior vertices are connected by a basic edge.
\end{lem}

\begin{proof}
Let $p$ be an interior fixed point of $\Gamma$. Let $e$ be a basic edge emanating from $p$ with endpoint $q$ and weight $\gamma$. Also let $\gamma'$ and $\gamma''$ be the weights of the other basic edges emanating from $p$ and $q$ and let $\alpha$, $\alpha'$ be the weights of the fiber edges emanating from $p$ and $q$. Through the connection we obtain
$\alpha=k\gamma+l\gamma'$ and $\alpha'=l\gamma''+m\gamma$ for some $k,l,m\in\mathbb{Z}$. As $p$ is an interior vertex we have $k,l<0$. This implies that the second vertex $q'$ in the fiber containg $q$ has emanating weights $\gamma''$, $-\gamma$, and $-l\gamma''-m\gamma$. Since $-l>0$ we deduce that $q'$ is an exterior vertex. We conclude that if the fiber contains an interior vertex, then $q$ is interior.
\end{proof}

\begin{lem}\label{lem:convexpath}
Let $p$ be an interior vertex of $\Gamma$. Further let $e,e'$ be the oriented basic edges emanating from $p$, let $q$ be the endpoint of $e'$ and $e_F$ be the fiber edge emanating from $q$. Then there is no $2$-valent GKM subgraph of $\Gamma$ that contains $e$, $e'$, and $e_F$.
\end{lem}

\begin{proof}
We assume the existence of a compatible connection on such a subgraph which implies that the weights $\gamma$, $\gamma'$, and $\alpha$ of $e$, $e'$, and $e_F$ satisfy $\gamma\equiv \alpha\mod \gamma'$. Let $\beta$ be the weight of the fiber edge emanating from $p$ and observe that
\[\beta\equiv\alpha\equiv\gamma\mod\gamma'.\]
Since the weights of the edges emanating from $p$ are $\gamma$, $\gamma'$, and $\beta$, this implies that $p$ is not an interior vertex.
\end{proof}

\begin{prop}
The signed graph $\Gamma$ does not satisfy the criterion from Corollary \ref{cor:Tolmanextension}. In particular, it is not realized by a Hamiltonian action on a K\"ahler manifold.
\end{prop}

\begin{proof}
Let $p$ be an interior fixed point and consider the two basic edges emanating from $p$. Suppose that we find a polytope type $2$-valent GKM subgraph of $\Gamma$ (i.e.\ a closed convex loop) containing those basic edges. Moving along such a path, starting with a basic edge at $p$, there are two possible choices for continuing the path: the basic edge and the fiber edge. However by Lemma \ref{lem:convexpath} the choice has to be the basic edge. Consequently, in both directions we need to move along basic edges until we reach exterior vertices. By Lemma \ref{lem:connectinteriorvertex}, moving from an interior vertex into a fiber containing an interior vertex will end in the interior vertex. Since by the  assumption that $\Gamma$ has $n-1$ interior vertices there is only one fiber $F\subset \Gamma$ which does not contain an interior fixed point, it follows that the desired 2-valent polytope type subgraph would need to contain a lift of the closed path running around $B$ once, which starts and ends in the fiber $F$. Since $\Gamma$ is of twisted type this path is not closed. Again by Lemma \ref{lem:convexpath} it follows that the subgraph does not contain the fiber edge in $F$ so the only way to close the loop is to continue with the other lift of the loop around $B$ until we reach the starting point. Thus the only possible choice of $2$-valent subgraph is $\Gamma$ with all the fiber edges removed. This is indeed a signed GKM subgraph but globally it is not a polytope type graph. This follows from Lemma \ref{lem:winding number} below as the winding number -- which we define in the next section -- of the subgraph in question is $2$.
\end{proof}

\subsection{Case II}\label{subsec:case II}

Assume now that we have another signed GKM structure $\Gamma'$ on the underlying graph of $\Gamma$ such that the induced unsigned GKM graph agrees with that of $\Gamma$. Let $e_i$ denote the oriented edges of $B$ as in Section \ref{sec:graphs in dim 6} and let $f_i,h_i$ be the edges in $\Gamma$ over $e_i$ such that $f_i$ and $f_{i+1}$ are adjacent for $i=1,\ldots,n-1$, and the same holds for $f_n$ and $h_1$. We also denote by $g_i$ the directed edge in the fiber over $v_i$ whose starting point is the same as that of $f_i$. Occasionally we extend this notation $\mathrm{mod}~n$ such that e.g.\ $f_{n+1}=h_1$ and $f_0=h_n$.

In case II we furthermore make the assumption that in the signed structure of $\Gamma'$ the labels of $f_i$ and $h_i$ agree (where in full generality they might differ by a sign). We do not assume that $\Gamma'$ admits a signed fibration over $B$, but as the underlying graphs are unchanged we may still speak about basic and fiber edges in $\Gamma'$.

\begin{lem}
The weights of the fiber edges of $\Gamma'$ agree with those of $\Gamma$ up to a global sign.
\end{lem}

\begin{proof}
let $\alpha_{i}'$ be the weight associated to $g_i$ in $\Gamma'$ and let $\alpha_i$ be the weight associated to $g_i$ in $\Gamma$. Let $\nabla'$ denote a compatible connection of $\Gamma'$. For $i=1,\ldots n-1$, we claim that $\nabla_{f_i}'g_i=g_{i+1}$ or $\nabla_{h_i}'\overline{g_i}=\overline{g_{i+1}}$. If this were false, then we would have $\nabla_{f_i}'g_i=f_{i+1}$ and $\nabla'_{h_{i}}\overline{g_i}=h_{i+1}$. But by assumption $f_{i+1}$ and $h_{i+1}$ have the same weight $\pm\gamma_{i+1}$ so
\[\alpha_i'\equiv \pm\gamma_{i+1}\equiv -\alpha_i'\mod \gamma_i\]
which is a contradiction.
So the claim holds, which implies that if $\alpha_i=\alpha_i'$, then also $\alpha_{i+1}=\alpha_{i+1}'$ because
\[\alpha_{i+1}'\equiv\alpha_i'\equiv\alpha_i\equiv\alpha_{i+1}\mod\gamma_i.\]
Analogously $\alpha_i=-\alpha_i'$ implies $\alpha_{i+1}=-\alpha_{i+1}'$. Thus the signs of the $\alpha_i$ and $\alpha_i'$ either globally agree or globally disagree.
\end{proof}

\begin{lem}\label{lem:even signed structure}
If $n$ is odd, then $\Gamma\cong \Gamma'$. If $n$ is even, then either $\Gamma\cong \Gamma'$ or $\Gamma'$ is isomorphic to the signed structure that arises from $\Gamma$ by changing the signs of the labels of $f_{2i}$ and $h_{2i}$ for $i=1,\ldots,n/2$ where a compatible connection can be chosen identical to that of $\Gamma$.
\end{lem}

\begin{proof}
Let $\gamma_i'$ be the weight of $f_i$ and $h_i$ in $\Gamma'$. Assume that for some $i=1,\ldots,n-2$ we have $\gamma_i=\gamma_i'$ but $\gamma_{i+2}=-\gamma_{i+2}'$. Since $\gamma_i\equiv -\gamma_{i+2}\mod \gamma_{i+1}$ it follows that $\gamma_{i}'\not\equiv-\gamma_{i+2}'\mod \gamma_{i+1}$. As a consequence we must have $\nabla'_{f_{i+1}}f_i=g_{i+2}$ and $\nabla'_{h_{i+1}}h_i=\overline{g_{i+2}}$ which results in the contradiction
\[\alpha_{i+2}'\equiv \gamma_i'\equiv -\alpha_{i+2}'\mod \gamma_{i+1}'.\]
If we assume that $\gamma_i=-\gamma_i'$ but $\gamma_{i+2}=\gamma_{i+2}'$, then we arrive at the same contradiction. We have proved that $\gamma_i$ and $\gamma_{i+2}$ must either both agree or both disagree with their counterparts $\gamma_i'$ and $\gamma_{i+2}'$. The same holds for the pairs $\gamma_n$, $\gamma_2$ and $\gamma_n', \gamma_2'$. So if $n$ is odd the $\gamma_i'$ globally agree or disagree with the $\gamma_i$. In each case we have $\Gamma\cong \Gamma'$ due to the previous lemma. If $n$ is even then the sign of the odd or the even edges may be switched independently of the other. Still, globally changing the signs of all the $\gamma_i'$ or all $\alpha_i'$ yields isomorphic graphs so the Lemma follows.
\end{proof}

\begin{prop}\label{prop:CaseII nicht symplectic}
If $n\neq 4$ is even, then the alternative signed structure $\Gamma'$ as in Lemma \ref{lem:even signed structure} can not be realized as the GKM graph of a symplectic action.
\end{prop}

For the proof we will make use of the following concept.

\begin{defn}
Let $(w_1,\ldots,w_n)$ be a sequence of vectors in $\mathbb{R}^2$ and let $\varepsilon=\pm 1$. The \emph{winding number} of the sequence with respect to the orientation $\varepsilon$ is given by
\[\sigma(w_1,\ldots,w_n,\varepsilon)=\frac{1}{2\pi}\sum_{i=1}^n |\eta_i|\]
where $\eta_i$ is the angle between $w_i$ and $w_{i+1}$ with representative chosen in $[0,2\pi)$ if $\varepsilon=1$ and in $(-2\pi,0]$ if $\varepsilon=-1$ (where we set $w_{n+1}=w_1$).

\end{defn}

\begin{lem}\begin{enumerate}[(i)]\label{lem:winding number}
\item If $\gamma_i\in \mathbb{Z}_\mft^*$ are the weights along a path around a 2-valent signed GKM graph, then the angle between $\gamma_i$ and $\gamma_{i+1}$ (resp.\ between $\gamma_n$ and $\gamma_1$) is, for all $i$, either always represented in $(0,\pi)$ or always represented in $(-\pi,0)$ (we speak of a locally convex sequence). Choosing $\varepsilon$ such that the angles get measured in the respective interval minimizes the winding number. We call this the preferred orientation for $(\gamma_1,\ldots,\gamma_n)$.

\item If $(\gamma_1,\ldots,\gamma_n)$ is a locally convex sequence with preferred orientation $\varepsilon$ and $n$ is even,
then \[\sigma(\gamma_1,\gamma_3,\ldots,\gamma_{n-1},\varepsilon)=\sigma(\gamma_1,\gamma_2,\ldots,\gamma_n,\varepsilon).\]

\item Let $\gamma_1,\ldots,\gamma_n$ be the vectors along the boundary of a convex polytope in $\mathbb{R}^2$ with $n$ vertices. We have $\sigma(\gamma_1,\ldots,\gamma_n,\varepsilon)=1$ with respect to the preferred orientation $\varepsilon$. If $n$ is even, then the sequence $(\gamma_1,-\gamma_2,\ldots,\gamma_{n-1},-\gamma_{n})$ is locally convex with preferred orientation $-\varepsilon$ and
\[\sigma(\gamma_1,-\gamma_2,\ldots,\gamma_{n-1},-\gamma_{n},-\varepsilon)=\frac{n-2}{2}.\]

\end{enumerate}
\end{lem}

\begin{proof}
For the proof of $(i)$ we observe that the connection of the signed graph has to transport $\gamma_i$ onto $-\gamma_{i+2}$ along $\gamma_{i+1}$ so
\[\gamma_i\equiv -\gamma_{i+2}\mod \gamma_{i+1}.\]
We conclude that $\gamma_i$ and $\gamma_{i+2}$ lie on opposite sides of the ray defined by $\gamma_{i+1}$ as they can not lie on the ray due to the condition of adjacent weights being linearly independent. Thus if the angle between $\gamma_{i}$ and $\gamma_{i+1}$ is represented in $(0,\pi)$ (resp.\ in $(-\pi,0)$), then the same holds for the angle between $\gamma_{i+1}$ and $\gamma_{i+2}$.

Assertion $(ii)$ follows from the fact that two consecutive angles never add up to a full rotation and so nothing is lost by skipping every second vector.

For the proof of $(iii)$ we assume for simplicity that the preferred orientation is $\varepsilon=1$, i.e.\ $\eta_i\in (0,\pi)$ for $i=1\ldots,n$, where $\eta_i$ denotes the angle between $\gamma_i$ and $\gamma_{i+1}$ (and $\eta_n$ is the angle between $\gamma_n$ and $\gamma_1$). The remaining case is proved analogously. We have
\[\sum_{i=1}^n\eta_i=2\pi\]
Now note that the angle between $\gamma_i$ and $-\gamma_{i+1}$ as well as from $-\gamma_i$ to $\gamma_{i+1}$ is represented in $(-\pi,0)$. It follows that $(\gamma_1,-\gamma_2,\ldots,\gamma_{n-1},-\gamma_{n})$ is indeed locally convex with the opposite preferred orientation. Now by $(ii)$ we have
\[\sigma(\gamma_1,-\gamma_2,\ldots,\gamma_{n-1},-\gamma_{n},-1)=\sigma(\gamma_1,\gamma_3,\ldots,\gamma_{n-1},-1).\]
The angle between $\gamma_i$ and $\gamma_{i+2}$ with respect to this orientation is exactly $\eta_i+\eta_{i+1}-2\pi$ so we obtain
\[\sigma(\gamma_1,\gamma_3,\ldots,\gamma_{n-1},-1)=\frac{1}{2\pi}\sum_{i=1}^{n/2}2\pi-\eta_{2i}-\eta_{2i+1}=\frac{n-2}{2}\]
as claimed.
\end{proof}

\begin{proof}[Proof of Proposition \ref{prop:CaseII nicht symplectic}]
Let $B'$ be the signed GKM structure on the underlying graph of $B$ where $e_{2i+1}$ has the weight $\gamma_{2i+1}$ and $e_{2i}$ has the weight $-\gamma_{2i}$. Note that the connection of $B$ also is a compatible connection on $B'$ and that $\Gamma'\rightarrow B'$ is a signed GKM fibration. Let $(k_1,\ldots,k_n)\in\mathbb{Z}^n$ be the vector (unique up to sign) corresponding to the fibration
$\Gamma\rightarrow B$ in the sense of Section \ref{sec:graphs in dim 6}. Then $\Gamma'\rightarrow B'$ corresponds to $(-k_1,k_2,\ldots,-k_{n-1},k_n)$. As $\Gamma$ was assumed to have the maximal number of interior fixed points, it follows from Proposition \ref{prop:interior} that $\Gamma'$ has the minimal number of interior fixed points, i.e., exactly one.

Now assume that $\Gamma'$ is realized by a Hamiltonian action. The moment image is a convex polytope in $\mathfrak{t}^*$, which we identify with $\mathbb{R}^2$, spanned by the images of the exterior fixed points. Note that, as part of the convexity theorem for Hamiltonian actions, the preimage of the vertices of the moment image is connected and contains only fixed points. Thus the vertices of the polytope correspond bijectively to exterior vertices in the GKM graph $\Gamma'$. A path around the boundary of this polytope corresponds to a closed path in $\Gamma'$ that runs through every exterior vertex without going through a vertex twice or going through an interior vertex. Since the sequence of weights along this path correspond exactly to the slopes of the boundary edges of the polytope, we deduce that it has winding number equal to $1$ with respect to its preferred orientation.

Now assume without loss of generality that the unique interior fixed point of $\Gamma'$ is the end point of $g_i$ and that we have a path as above starting with the edge $f_i$. Since the path must not go through the interior fixed point the only possibility is the path \[f_i,g_{i+1},h_{i+1},\overline{g_{i+2}},\ldots,h_{i-2},\overline{g_{i-2}},f_{i-1}\]
which alternates between basic and fiber edges (excluding the fiber which contains the interior vertex). In case $i$ is odd, the associated sequence of weights is
\[
(\gamma_i,\alpha_i,-\gamma_{i+1},-\alpha_{i+1},\ldots,\gamma_{i-2},-\alpha_{i-1},-\gamma_{i-1}),
\] and then by Lemma \ref{lem:winding number} its winding number with respect to the preferred orientation $\varepsilon$ satisfies
\[\sigma(\gamma_i,\alpha_i,-\gamma_{i+1},-\alpha_{i+1},\ldots,-\alpha_{i-1},-\gamma_{i-1},\varepsilon)=\sigma(\gamma_i,-\gamma_{i+1},\ldots,\gamma_{i-2},-\gamma_{i-1},\varepsilon)\geq\frac{n-2}{2},\]
with equality on the right if $\varepsilon$ is the preferred orientation for the central expression. If $n>4$ then this is in any case not equal to $1$ which is a contradiction. In case $i$ is even, in the weight sequences the signs of all $\gamma_k$ are reversed, whence we arrive at the same contradiction.
\end{proof}

\subsection{Case III}

Let $f_i,h_i$ and $g_i$ denote the basic and fiber edges of the underlying graph of $\Gamma$ as in the previous section. It remains to treat the case where the weights of $f_i$ and $h_{i}$ do not agree for some $i$. Let $\Gamma'$ be a signed GKM structure on the unsigned GKM graph of $\Gamma$ satisfying the above property.

\begin{lem}\begin{enumerate}[(i)]\label{lem:twistedlem}
\item The weights of $f_i$ and $h_i$ do not agree for any $i=1,\ldots,n$.
\item The vector $(k_1,\ldots,k_n)$ corresponding to the original signed fibration $\Gamma\rightarrow B$ in the sense of Section \ref{sec:graphs in dim 6} satisfies $k_i=\pm 1$ for $i=1,\ldots,n$.
\item If $k_i$ and $k_{i+1}$ have the same sign, then the weights of $f_i$ and $f_{i+1}$ in the signed structure $\Gamma'$ both agree or both disagree with the weights in $\Gamma$.
\end{enumerate}
\end{lem}

\begin{proof}
Let $\nabla'$ be a connection that is compatible with the signed structure $\Gamma'$ and let $\gamma_i',\delta_i'=\pm\gamma_i$ be the weights of $f_i$ and $h_i$. If $\gamma_i'=-\delta_i'$ then necessarily $\nabla'_{g_{i+1}}\overline{f_i}=h_{i+1}$ and $\nabla'_{\overline{g_{i+1}}}\overline{h_i}=f_{i+1}$. This implies
\begin{equation}\label{eq:caseIII1}\gamma_{i+1}'\equiv -\delta_{i}'\equiv\gamma_i'\equiv-\delta_{i+1}'\mod\alpha_{i+1}\end{equation}
and assertion $(i)$ follows.

For the original signed structure $\Gamma$, the equation $\gamma_i'\equiv -\delta_{i+1}'\mod \alpha_{i+1}$ in \eqref{eq:caseIII1} translates to $\gamma_i=\pm \gamma_{i+1}\mod \alpha_{i+1}$. This congruence in the lattice $\mathbb{Z}_\mft^*$ can only be solved if up to sign we have $\alpha_{i+1}=\gamma_i\pm\gamma_{i+1}$ because $\gamma_i$ and $\gamma_{i+1}$ form a basis of $\mathbb{Z}_\mft^*$. In particular $k_i,k_{i+1}=\pm 1$ is necessary which proves $(ii)$.

More specifically, if $k_i$ and $k_{i+1}$ have the same sign then $\pm\alpha_{i+1}=\gamma_i-\gamma_{i+1}$ and the only solvable congruence of the above form is $\gamma_i=\gamma_{i+1}\mod\alpha_{i+1}$. Thus if $\gamma_i'=\gamma_i$ then $\gamma'_i\equiv \gamma_{i+1}'\mod\alpha_{i+1}$ (see \eqref{eq:caseIII1}) implies $\gamma_{i+1}'=\gamma_{i+1}$. Analogously $\gamma_i'=-\gamma_i$ implies $\gamma_{i+1}'=-\gamma_{i+1}$.
\end{proof}

By assumption $\Gamma$ has the maximal number of interior fixed points in the sense of Section \ref{sec:graphs in dim 6} and thus there is only one spot $j\in\{1,\ldots,n\}$ for which $k_j\neq k_{j-1}$ (setting $k_0=-k_n$). Fixing this $j$, Lemma \ref{lem:twistedlem} implies that regarding basic edges $\Gamma'$ has the form
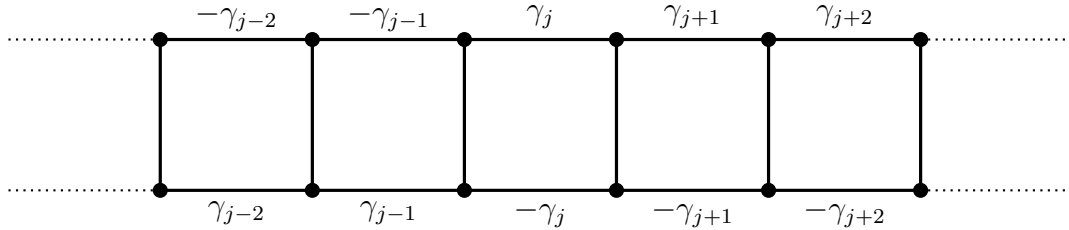
\begin{center}
\begin{tikzpicture}

\draw[very thick] (-7,1) -- ++(4,0) -- ++(0,-2) --++(-4,0)--++(0,2);
\draw[very thick] (-5,1) -- ++(0,-2);
\draw[very thick] (-3,1) -- ++(2,0);
\draw[very thick] (-3,-1) -- ++(2,0);
\draw[very thick] (-1,1) -- ++ (4,0) -- ++(0,-2) --++(-4,0) --++(0,2);
\draw[very thick] (1,1) --++(0,-2);
\draw[dotted, thick] (3,1)--++(2,0);
\draw[dotted, thick] (3,-1)--++(2,0);
\draw[dotted, thick] (-7,1)--++(-2,0);
\draw[dotted, thick] (-7,-1)--++(-2,0);

\node at (-6,1.3){$-\gamma_{j-2}$};
\node at (-6,-1.3){$\gamma_{j-2}$};
\node at (-4,1.3){$-\gamma_{j-1}$};
\node at (-4,-1.3){$\gamma_{j-1}$};
\node at (-2,1.3){$\gamma_{j}$};
\node at (-2,-1.3){$-\gamma_{j}$};
\node at (0,1.3){$\gamma_{j+1}$};
\node at (0,-1.3){$-\gamma_{j+1}$};
\node at (2,1.3){$\gamma_{j+2}$};
\node at (2,-1.3){$-\gamma_{j+2}$};

%\node at (-6.6,0){$\alpha_{1$};
%\node at (-4.6,0){$\alpha_{2}$};
%\node at (-2.6,0){$\alpha_{3}$};
%\node at (-0.6,0){$\alpha_{n-1}$};
%\node at (1.4,0){$\alpha_{n}$};
%\node at (3.4,0){$\alpha_{n+1}$};

\node at (-7,-1)[circle,fill,inner sep=2pt]{};
\node at (-7.3,-1.3){};

\node at (-7,1)[circle,fill,inner sep=2pt]{};
\node at (-7.3,1.3){};

\node at (-5,1)[circle,fill,inner sep=2pt]{};
\node at (-5,1.3){};

\node at (-5,-1)[circle,fill,inner sep=2pt]{};
\node at (-5,-1.3){};

\node at (-3,1)[circle,fill,inner sep=2pt]{};
\node at (-3,1.3){};

\node at (-3,-1)[circle,fill,inner sep=2pt]{};
\node at (-3,-1.3){};

\node at (-1,1)[circle,fill,inner sep=2pt]{};
\node at (-1,1.3){};

\node at (-1,-1)[circle,fill,inner sep=2pt]{};
\node at (-1,-1.3){};
\node at (1,1)[circle,fill,inner sep=2pt]{};
\node at (1,1.3){};

\node at (1,-1)[circle,fill,inner sep=2pt]{};
\node at (1,-1.3){};
\node at (3,1)[circle,fill,inner sep=2pt]{};
\node at (3,1.3){};

\node at (3,-1)[circle,fill,inner sep=2pt]{};
\node at (3,-1.3){};

\end{tikzpicture}
\end{center}
where the horizontal edges are oriented from left to right and the ends (the not depicted positions $1$ and $n+1$) of the ladder are glued in a twisted fashion. Note that interchanging the top and bottom row while applying multiplication with $-1$ in $\mathbb{Z}_\mft^*$ defines an automorphism of $\Gamma'$. Consequently we can assume that the $f_i$ are the horizontal edges in the upper row, the $h_i$ are the horizontal edges in the lower row, and the $g_i$ are the vertical edges emanating from the starting point of $f_i$.

\begin{prop}
In the graph above there is a unique way to define the signs of the weights of the $g_i$ such that there exists a compatible connection. Consequently the signed GKM structure $\Gamma'$ is unique up to isomorphism. Furthermore $\Gamma'$ has only interior fixed points and is not realized by a Hamiltonian action.
\end{prop}

\begin{proof}
Let $\alpha_i'$ denote the weight of $g_i$ in $\Gamma'$. By the choice of $j$ we have $\pm\alpha_j'=\gamma_{j-1}+\gamma_j$ and $\pm\alpha_i'=\gamma_{i-1}-\gamma_i$ for the remaining values $i\neq j$. Observe that since $\gamma_{j-2}\equiv-\gamma_j\mod \gamma_{j-1}$ we cannot have $\nabla'_{f_{j-1}}\overline{f_{j-2}}=f_j$ due to the change of sign in $\Gamma'$. As a consequence we have $\nabla'_{f_{j-1}}\overline{f_{j-2}}=g_j$ which implies $\alpha_j'\equiv\gamma_{j-2}\equiv -\gamma_j\mod\gamma_{j-1}$ and forces $\alpha_j'=-\gamma_{j-1}-\gamma_j$. Analogously one has $\nabla'_{f_j}\overline{f_{j-1}}=g_{j+1}$ and it follows that $\alpha_{j+1}'\equiv\gamma_{j-1}\equiv-\gamma_{j+1}\mod\gamma_j$ forcing $\alpha_{j+1}'=\gamma_j-\gamma_{j+1}$.

We prove inductively that $\alpha_i'=\gamma_{i-1}-\gamma_i$ for $i=j+1,\ldots,n$ and $\alpha_i'=-\gamma_{i-1}+\gamma_i$ for $i=1,\ldots,{j-1}$. We showed this already for $j+1$ and we assume it holds for some $i\in \{j+1,\ldots,n-1\}$. Through the congruence $\alpha_i'\equiv \gamma_{i-1}\equiv -\gamma_{i+1}\not\equiv\gamma_{i+1}\mod\gamma_i$ we see that $\nabla'_{f_i}g_i\neq f_{i+1}$. Consequently, $\nabla_{f_i}g_i=g_{i+1}$ which implies $\alpha'_{i+1}\equiv \gamma_{i-1}-\gamma_i\equiv -\gamma_{i+1}\mod \gamma_i$ and thus $\alpha'_{i+1}=\gamma_i-\gamma_{i+1}$. The rest of the argument is carried out analogously where one first shows that $\nabla'_{f_n}g_n=\overline{g_1}$ and thus $\alpha_1'\equiv -\gamma_{n-1}+\gamma_n\equiv \gamma_1\mod\gamma_n$ due to the twist in our notation. It follows that $\alpha_1'=-\gamma_n+\gamma_1$ and from there on the induction can be continued up to $j-1$.

We have proved that $\Gamma'$ is unique up to isomorphism. One easily checks that a compatible connection $\nabla'$ is indeed given as follows: $\nabla'_{f_i}$ and $\nabla'_{h_i}$ send basic to basic edges except for $i=j-1,j$ where basic and fiber edges are interchanged. Along the $g_i$, the basic edges $\overline{f_i}$ and $f_{i+1}$ get transported to $h_{i+1}$ and $\overline{h_{i}}$. Clearly, all vertices in $\Gamma'$ are interior. In particular it does not have a linear realization in the sense of Section \ref{subsec:introsymplec} and can thus not come from a Hamiltonian action.
\end{proof}

\section{Distinguishing the equivariant homotopy type}

The goal of this section is to show that the equivariant homotopy type of the previous constructions in general depends on the input data. In particular, it will follow that the previously developed methods produce infinite families of pairwise not equivariantly homotopy equivalent examples.

\begin{prop}\label{prop:isooftotalspaces}
Let $\Gamma\rightarrow B$ and $\Gamma'\rightarrow B'$ be two fiberwise signed GKM fibrations as in Section \ref{sec:graphs in dim 6} such that the base graphs have at least $5$ vertices. Then the GKM graphs $\Gamma$ and $\Gamma'$ are isomorphic if and only if the following hold
\begin{itemize}
\item There is a GKM isomorphism $\varphi\colon B\rightarrow B'$.
\item If we fix the data for $B$ as in Remark \ref{rem:correspondencedata} needed to define the correspondence in Proposition \ref{prop:correspondence} and use $\varphi$ to fix the corresponding choices for $B'$, then the elements in $((\mathbb{Z}-0)^n/\pm)\times\{0,1\}$ associated to the fibrations coincide.
\end{itemize}
\end{prop}

\begin{rem}
Given $\Gamma\rightarrow B$ as above then as stated before the map in Proposition \ref{prop:correspondence} depends on a fixed enumeration of the vertex set of $B$ and on choices of signs for the first two edges. If we change this data by choosing a different sign for the first or second weight, then the fibration associated to $([k_1,\ldots,k_n],\eta)$ will now correspond to $([-k_1,k_2,\ldots,(-1)^nk_n],\eta)$. Changing the enumeration of the underlying $n$-gon, the $k_i$ get permuted by the corresponding permutation of the dihedral group. Note however that on top of the permutation some additional signs will appear. We leave the details of the exact signs to the interested reader and settle for the slightly suboptimal corollary below.
\end{rem}

\begin{cor}\label{cor:distinguish eq HT}
Let $\Gamma\rightarrow B$ and $\Gamma'\rightarrow B$ be GKM fibrations associated to $([k_1,\ldots,k_n],\eta)$ and $([k_1',\ldots,k_n'],\eta')$. If $n\geq 5$, then in order for geometric realizations of $\Gamma$ and $\Gamma'$ to be equivariantly homotopy equivalent, it is necessary that $\eta=\eta'$ and $[k_1,\ldots,k_n]=[k_1',\ldots,k_n']$ up to signs and permutations from the dihedral group.
\end{cor}

\begin{proof}
It follows from Proposition \ref{prop:isooftotalspaces} as well as the subsequent remark that the conditions in the corollary are necessary in order for $\Gamma$ and $\Gamma'$ to be isomorphic. It is shown in \cite{FranzYamanaka} that realizations of non-isomorphic graphs have non-isomorphic equivariant cohomology algebras which implies the claim.
\end{proof}

\begin{proof}[Proof of Proposition \ref{prop:isooftotalspaces}]
If $\Gamma$ and $\Gamma'$ are isomorphic then in particular they have the same number of vertices. Thus the underlying graphs of $B$ and $B'$ are both $n$-gons for some $n\geq 5$. We argue that an isomorphism $\tilde{\varphi}\colon\Gamma\cong \Gamma'$ has to respect the decomposition into horizontal and vertical edges of the respective fibrations. To see this note that a horizontal edge $e$ has the following property: there is another edge $e'$ (namely the other horizontal edge over the same edge in the base) such that after removing $e$ and $e'$ the shortest path between $i(e)$ and $t(e)$ has length at least $n-1\geq 4$. On the other hand any vertical edge $e$ has the property that, after removing $e$ and any other edge $e'$, the shortest path between $i(e)$ and $t(e)$ has length $3$. Both properties are respected by graph isomorphisms so $\varphi$ indeed respects the decomposition of fiber and basic edges.

It follows that the lift of a path around the $n$-gon $B$ gets mapped by $\tilde{\varphi}$ to the lift of a path around the $n$-gon $B'$. Thus in particular $\Gamma$ and $\Gamma'$ must either be both of twisted type or of product type so $\eta=\eta'$. Since this is true for both possible lifts and their images under  $\tilde{\varphi}$ are connected through fiber edges it follows that $\tilde{\varphi}$ respects pairs of basic edges. This implies that lifting an edge from $B$ to $\Gamma$, mapping it to $\Gamma'$ and pushing it down to $B'$ induces a well defined graph isomorphism $\varphi\colon B\rightarrow B'$.
Since $\tilde{\varphi}$ is a GKM isomorphism there is some automorphism $\psi$ of $\mathbb{Z}_\mft^*$ such that for any edge $e$ in $\Gamma$ we have $\psi(\alpha(e))=\alpha'(\tilde{\varphi}({e}))$, where $\alpha$ and $\alpha'$ denote the axial functions of $\Gamma$ and $\Gamma'$. It follows that if $\tilde{e}$ is an edge in $\Gamma$ over some edge $e\in E(B)$ then \[\psi(\alpha_B(e))=\psi(\alpha(\tilde{e}))=\alpha'(\tilde{\varphi}(\tilde{e}))=\alpha_{B'}(\varphi(e)).\]
Thus $\varphi$ is a GKM isomorphism.

We enumerate the vertices and edges of $B$ and choose weights $\gamma_i$ as in Remark \ref{rem:correspondencedata}. We give $B'$ the enumeration induced by $\varphi$ and choose signs of $\gamma_i'$ such that $\psi(\gamma_i)=\gamma_i'$. If the orientations of the fiber edges are chosen compatibly with $\tilde{\varphi}$ then the corresponding weights $\alpha_i$ and $\alpha_i'$ and the resulting integers $k_i$  and $k_i'$ as in Proposition \ref{prop:correspondence} satisfy
\[k_i'\gamma_{i-1}'-k_{i-1}'\gamma_i'=\alpha_i'=\psi(\alpha_i)=k_i\psi(\gamma_{i-1})-k_{i-1}\psi(\gamma_i).\]
Thus $k_i=k_i'$ which proves one direction of the proposition. Conversely one easily checks that given $\varphi\colon B\rightarrow B'$ satisfying the conditions of the proposition, any graph isomorphism $\tilde{\varphi}\colon \Gamma\rightarrow \Gamma'$ covering $\varphi$ is a GKM isomorphism.
\end{proof}

\section{Cohomology and characteristic classes}
\label{sec:charclasses}
This section is devoted to compute the cohomology ring as well as the Chern classes of the realizations of the GKM fibrations. All cohomology rings in this section are with respect to integer coefficients.

First we remind the reader how the Chern classes of the total space of a projectivized bundle are computed in terms
of the Chern classes of the bundle and the base. Note that
usually this is done in case the bundle $E \to X$ is a holomorphic vector bundle over some complex
manifold $X$. But the same computations work in case of a $4$-dimensional quasitoric base manifold or $S^4$.

Assume that $X$ is a $4$-dimensional stably almost complex manifold and $E \to X$ a complex
vector bundle of rank $2$, and denote by $\pi\colon \mathbb{P}(E) \to X$ the canonical projection.
The vertical distribution $V\subset T \mathbb{P}(E)$ is a complex vector bundle (cf.\ Section
\ref{sec: sacs}), thus $\mathbb{P}(E)$ has a stable almost complex structure induced by the
decomposition $T \mathbb{P}(E) = V \oplus\pi^\ast(TX)$. Let $L \to \mathbb{P}(E)$ be the relative
tautological bundle, i.e., it restricts to every $\CC\PP^1$-fiber of $\pi \colon \mathbb{P}(E) \to X$
to the tautological bundle over this fiber. From \cite[p. 270
(20.7)]{MR658304} we have that the cohomology ring $H^*(\mathbb{P}(E))$ is the quotient
ring of the polynomial ring $H^*(X)[x]$ by the ideal $I$ generated by
\[
  x^2 +c_1(E)x + c_2(E)
\]
where $x := c_1(\overline L) \in H^*(\mathbb{P}(E))$ and $\overline L$ denotes the dual bundle of
$L$. We obtain

\begin{prop}\label{prop:Chern classes projec}
  The Chern classes of the stable almost complex structure of $\mathbb{P}(E)$ in the ring
   $H^\ast(X)[X]/I$ are given by
  \[
    c_1(\mathbb P(E)) = c_1(X) + c_1(E) + 2x
  \]
  and
  \[
    c_2(\mathbb P(E)) = c_2(X)
    + c_1(E)c_1(X)+2 c_1(X)x.
  \]
  Clearly $c_3(\mathbb P(E))$ is determined by the Euler characteristic of $\mathbb{P}(E)$
  which is equal twice the Euler characteristic of $X$.
\end{prop}
\begin{proof}
The relative Euler sequence \cite[Remark 2.4.5]{MR2093043} holds also in this setting,
  i.e. we have a short exact sequence of complex vector bundles
\[
  0 \longrightarrow \underline \CC \longrightarrow \pi^*(E) \otimes \overline L
  \longrightarrow V \longrightarrow 0,
\]
where $\underline\CC$ is the trivial vector bundle. Let $c$ denote the total Chern class,
  then $c(\mathbb{P}(E)) = \pi^*(c(X))c(V)$ in $H^\ast(\mathbb{P}(E))$. From the relative
  Euler sequence we infer $c(V) = c(\pi^{*}(E) \otimes \overline L)$ and using the splitting
  priciple we obtain for the tensor product
  \[
    c_1(\pi^*(E)\otimes \overline L) = \pi^\ast(c_1(E)) + 2c_1(\overline L).
  \]
  as well as
  \[
    c_2(\pi^\ast(E) \otimes \overline L) = \pi^\ast(c_2(E)) + \pi^\ast(c_1(E))c_1(\overline L)
    +c_1(\overline L)^2.
  \]
  Identifying now $H^\ast(\mathbb{P}(E))$ with $H^\ast(X)[x]/I$ we compute
  \[
    c_1(\mathbb{P}(E)) = c_1(X) +c_1(E) +2x
  \]
  and
  \begin{align*}
    c_2(\mathbb{P}(E)) &= c_1(E)c_1(X) + 2c_1(X)x + c_2(E) + c_1(E)x +x^2 +c_2(X)\\
    &= c_1(E)c_1(X) + 2c_1(X)x + c_2(X),
  \end{align*}
  where we used that $c_2(E) + c_1(E)x + x^2$ is zero in $H^\ast(X)[x]/I$.
\end{proof}

We suppose $X$ is a quasitoric manifold of dimension $4$ o $S^4$ and $T$ the $2$-torus acting on $X$.
We use the notation from Remark \ref{rem:correspondencedata} for the GKM graph of $X$. We denote by $\delta_i$ the element of $\bigoplus_{i=1}^n
H^\ast(BT) = H^\ast_T(X^T)$ which is zero, except at $v_i$, where it is equal
to $1 \in H^\ast(BT)$. We extend the notation to $\delta_{i+n}=\delta_{i}$. From \cite[Theorem 7.7]{MasudaPanov} we infer that the
equivariant cohomology of $H_T^\ast(X)\subset H_T^\ast(X^T)$ is generated by
\begin{align*}
  \beta_i&:=-\gamma_{i-1}\delta_i + \gamma_{i+1}\delta_{i+1}
\end{align*}
for $i=1,\ldots,n$ (the $\beta_i$ are, in the language of \cite[Section 6.2]{MasudaPanov}, the Thom classes of the two-dimensional submanifolds corresponding to the edges in the GKM graph of $X$).

\begin{lem}\label{L:First Chern class}
Let $E \to X$ be a $T$-equivariant complex vector bundle of rank $2$ satisfying condition $(i)$ of Theorem \ref{thm:vectorbundle} (for some $(a_1,\ldots,a_n),(k_1,\ldots,k_n)\in \mathbb{Z}^n$, $\eta\in\{0,1\}$) and
denote by $c_i^T(E) \in H_T^\ast(X;\ZZ)$ the integral $i$-th equivariant Chern class of $E \to X$.
The image of $c_1^T(E)$ under the inclusion map $H_T^\ast(X) \to H_T^\ast(X^T)$ is
given by
\[
  \sum_{i=1}^{n} (k_i-2a_i)\beta_i
\]
  and that of $c_2^T(E)$  by
  \[
    \sum_{i=1}^n (2a_ia_{i-1}-a_{i-1}k_i-a_ik_{i-1})\beta_{i-1}\beta_i + \sum_{i=1}^{n} (a_i^2 -a_ik_i)\beta_i^2
  \]
where we set $a_{0} = (-1)^{\varepsilon_1} a_n+\eta k_0$, $k_{0}=(-1)^{\varepsilon_1+\eta} k_n$ and $\beta_{0}=(-1)^{\varepsilon_1}\beta_n$.
\end{lem}
\begin{proof}
  The image of the total Chern class $c^T(E)$ in $H_T^\ast(X^T)$ is given by (see \cite[Proposition 5.3]{1903.11684v1})
  \[
    \sum_{i=1}^{n} (1+\alpha_{i1})(1+\alpha_{i2})\delta_i,
  \]
  where $\alpha_{ij}$ are the weights of the $T$-representation on the fiber over the fixed point $v_i$. By abuse of notation we write also $c^T(E) \in H_T^\ast(X^T,\ZZ)$
  for the image of $c^T(E)$ under the inclusion homomorphism.
  By assumption the weights $\alpha_{ij}$ over $v_{i}$ are given by
  \[
    \alpha_{i1} = a_i\gamma_{i-1} -a_{i-1}\gamma_{i},\quad
    \text{and}
    \quad
    \alpha_{i2} = (a_i-k_{i})\gamma_{i-1} + (-a_{i-1}+k_{i-1})\gamma_{i}.
  \]
  and therefore
\[
  c_1^T(E) = \sum_{i=1}^{n} \left( (2a_i-k_i)\gamma_{i-1} + (-2a_{i-1}+k_{i-1})\gamma_i \right)\delta_i.
\]
  Using that $k_0 = (-1)^{\eta+\varepsilon_1} k_n$ (c.f. Section \ref{sec:constr-sec}) we compute
  \begin{align*}
    &c_1^T(E)\\ &= (2a_1-k_1)\gamma_0\delta_1 + (-2((-1)^{\varepsilon_1}a_n+\eta k_0)+k_0)\gamma_1\delta_1 + \sum_{i=2}^{n} \left( (2a_i-k_i)\gamma_{i-1} + (k_{i-1}-2a_{i-1})\gamma_i \right)\delta_i \\
    &=(2a_1-k_1)\gamma_0\delta_1 - (2a_n -(2\eta + (-1)^{\eta})k_n)\gamma_{n+1}\delta_{n+1} + \sum_{i=2}^{n} \left( (2a_i-k_i)\gamma_{i-1} + (k_{i-1}-2a_{i-1})\gamma_i \right)\delta_i \\
    &=
    (2a_1-k_1)\gamma_0\delta_1 - (2a_n-k_n)\gamma_{n+1}\delta_{n+1} + (2a_n-k_n)\gamma_{n-1}\delta_n
    - (2a_1-k_1)\gamma_2\delta_2 - \sum_{i=2}^{n-1} (2a_i-k_i)\beta_i\\
    &= \sum_{i=1}^{n} (k_i -2a_i)\beta_i.
  \end{align*}
  The second Chern class is given by
    \begin{align*}
      c_2^T(E) &= \sum_{i=1}^{n} (a_i\gamma_{i-1} - a_{i-1}\gamma_i)((a_i-k_i)\gamma_{i-1} +(-a_{i-1}+k_{i-1})\gamma_i)\delta_i \\
      &= \sum_{i=1}^{n} \left( (a_i^2 - a_ik_i)\gamma_{i-1}^2   + (-2a_ia_{i-1}+a_{i-1}k_i+a_ik_{i-1})\gamma_{i-1}\gamma_i
      +(a_{i-1}^2 - a_{i-1}k_{i-1})\gamma_i^2\right)\delta_i.
    \end{align*}
    Let us examine the middle term first. We have
      \begin{align*}
        \sum_{i=1}^n (-2a_ia_{i-1}+a_{i-1}k_i+a_ik_{i-1})\gamma_{i-1}\gamma_i\delta_i=
\sum_{i=1}^n (2a_ia_{i-1}-a_{i-1}k_i-a_ik_{i-1})\beta_{i-1}\beta_i
      \end{align*}
where $\beta_0:=-\gamma_{-1} \delta_n+\gamma_1 \delta_1=-(-1)^{\varepsilon_1}\gamma_{n-1}\delta_n+(-1)^{\varepsilon_1}\gamma_{n+1}\delta_n=(-1)^{\varepsilon_1}\beta_n$.
      The remaining terms can be rearranged as follows
      \begin{align*}
        \sum_{i=1}^{n} &\left( (a_i^2-a_ik_i)\gamma_{i-1}^2 + (a_{i-1}^2 - a_{i-1}k_{i-1})\gamma_i^2  \right)\delta_i\\
      &=\sum_{i=1}^{n} (a_i^2 -a_ik_i)\gamma_{i-1}^2 \delta_i +  \sum_{i=1}^{n-1} (a_i^2 -a_ik_i)\gamma_{i+1}^2 \delta_{i+1}
        +(a_0^2-a_0k_0)\gamma_1^2\delta_1 \\
        %&=\sum_{i=2}^{n-1} (a_i^2-a_ik_i)\beta_i^2+ (a_1^2 -a_1k_1)\gamma_n^2\delta_1
        %+ (a_n^2-a_nk_n)\gamma_{n-1}^2\delta_n + (a_1^2 -a_1k_1)\gamma_2^2\delta_2
        %+ (b_1^2+b_1k_0)\gamma_1^2\delta_1\\
        &=\sum_{i=1}^{n-1} (a_i^2-a_ik_i)\beta_i^2 +  (a_n^2-a_nk_n)\gamma_{n-1}^2\delta_n
        + (a_0^2-a_0k_0)\gamma_1^2\delta_1 \\
        &= \sum_{i=1}^{n} (a_i^2 - a_ik_i)\beta_i^2
      \end{align*}
    where in the last equality we used that $a_0^2-a_0k_0 = a_n^2+(-1)^{\epsilon_1}a_nk_0 (2\eta-1) = a_n^2 -a_nk_n$.
\end{proof}

The cohomology $H^\ast(X)$ is isomorphic to $H^\ast_T(X)/\left( H^{>0}(BT) \cdot H^\ast_T(X)  \right)$, see the proof of \cite[Lemma 2.1]{MasudaPanov}. We denote by $\overline \beta_i$ the elements in $H^\ast(X)$ which are the images of $\beta_i$ under the projection
map $H_T^\ast(X;\ZZ) \to H^*(X)$. Thus the elements $\overline\beta_i$ generate $H^*(X)$.

When it comes to computing the cohomology of the projectivization of the vector bundles from Theorem \ref{thm:vectorbundle} we note that the result does not depend on $a=(a_1,\ldots,a_n)$: the cohomology is completely determined by the GKM graph, on which $a$ has no effect. Thus we may set $a=0$ (note that then $a_0 = \eta k_0$) and apply \ref{L:First Chern class} to obtain

\begin{cor}\label{C:Cohomology of Pr}
  Let $E \to X$ be a $T$-equivariant complex
  vector bundle of rank $2$ as in Theorem \ref{thm:vectorbundle}. Then we have
  \[
    H^*(\mathbb{P}(E)) = H^*(X)[x]/\langle x^2 + \left(\sum_{i=1}^{n} k_i\overline\beta_i\right)x + \eta k_nk_1 \overline \beta_1 \overline\beta_n  \rangle
  \]
\end{cor}

Any $T$-invariant (stably) almost complex structure on $X$ induces such a structure on $\mathbb{P}(E)$, so that the decomposition $T \mathbb{P}(E)
= V \oplus \pi^*(TX)$ is a decomposition of stable almost complex vector bundles, cf.
 Theorem \ref{thm:mainthm} and Section \ref{sec: sacs}. Recall that at least a stably almost
 complex structure always exists in case $X$ is quasi toric or $S^4$. We wish to compute the Chern classes
 of the resulting structure on $\mathbb P(E)$ using Proposition \ref{prop:Chern classes projec}.
 Therefore we first have to determine the Chern classes of $X$, which we will conduct using
 equivariant cohomology. For simplicity we will restrict to the case of an almost complex
 structure on $X$ and the choice of the $\gamma_i$ will be assumed to be that of the resulting
 signed GKM structure. In this case the equivariant Chern class of $X$ is given by
 \[
   c^T(X) = \sum_{i=1}^{n} (1-\gamma_{i-1})(1+\gamma_i)\delta_i.
 \]
 An easy computation shows that
 \[
   c_1^T(X) = \sum_{i=1}^{n} \beta_i
 \]
 and
 \[
   c_2^T(X) = \sum_{i<j}^{} \beta_i\beta_j.
 \]
 When it comes to the vector bundles from Theorem \ref{thm:vectorbundle}, note that we may again assume $a=0$ for the computation since the Chern classes depend only on the GKM graph. In total we obtain

\begin{prop}\label{prop: chern classes of projectivization}
 If $X$ is almost complex and $E\rightarrow X$ is as in Theorem \ref{thm:vectorbundle}, then the Chern classes of $\mathbb{P}(E)$ expressed in the Chow ring are given by
  \[
    c_1(\mathbb{P}(E)) = \sum_{i=1}^{n} (k_i+1)\overline\beta_i +2x
  \]
and
\[
  c_2(\mathbb{P}(E)) = \sum_{i<j}^{} \overline\beta_i\overline\beta_j + \sum_{i,j}^{} k_i\overline\beta_i\overline\beta_j + 2 \sum_{i=1}^{n} \overline\beta_i x
\]
\end{prop}

Finally we would like to show in a special case, that the homotopy type of $\mathbb{P}(E)$
depends on $k$. We will consider the discriminant
of the symmetric trilinear form given by the triple cup product on integer cohomology in dimension
$2$, cf. \cite[Section 3.1 and 5.2]{MR1365849}: Let $N$ be a
simply-connected, closed and orientable $6$-manifolds such that its second Betti number is equal to
$2$. Choose a basis $(e_1,e_2)$ of $H^2(N)$ and an orientation homology class $[N] \in
H_6(N)$. Consider the following integers using the cup product of $H^*(N)$
\[
  n_0 := \langle e_1^3,[N]  \rangle,\quad n_1:=\langle e_1e_2^2, [N] \rangle,\quad
  n_2 := \langle e_1^2e_2, [N] \rangle,\quad n_3 := \langle e_2^3, [N] \rangle.
\]
The number
\[
  \Delta_N := (n_0n_3-n_1n_2)^2 - 4(n_0n_2 -n_1^2)(n_1n_3 - n_2^2)
\]
is invariant under the action of $\textrm{GL}(2,\ZZ)$ on $H^2(N)$ and does not depend on the chosen
orientation. Thus it represents an invariant of the homotopy type of $N$.

Let us now assume $X = \CC\PP^{2}$ with the standard action of $T^2$ and denote by $B$ the GKM graph of $X$.
Consider a signed GKM fibration $\Gamma\rightarrow B$ corresponding to $(k_1,k_2,k_3,\eta)$ in the sense of \ref{prop:correspondence} (with respect to some choice of data in $B$ as in Remark \ref{rem:correspondencedata}). We assume that the fibration is of twisted type i.e.\ $\eta=1$.
Let $E_k$ be a $T^2$-equivariant complex vector bundle as in Theorem \ref{thm:vectorbundle} with
$(a_1,\ldots,a_n)=0$ such that $\mathbb P(E)\rightarrow X$ realizes $\Gamma\rightarrow B$ and
denote
by $\mathbb{P}_k$ the projectivization of $E_k$. Note that in $H^*(\CC\PP^{2})$ we have
$\overline \beta_1= \overline\beta_2=\overline\beta_3$ and $\overline\beta_1\overline\beta_2 =
\overline\beta_2\overline\beta_3=\overline\beta_1\overline\beta_3$. From
\cite[Proposition 17]{MR1365849} it follows that
\[
  \Delta_{\mathbb{P}_k} = c_1(E_k)^2 -4c_2(E_k)
\]
when interpreting the Chern classes in $H^*(\mathbb{CP}^2)$ as integers. Thus using Lemma \ref{L:First Chern class} we obtain
\[
  \Delta_{\mathbb{P}_k} = (k_1+k_2+k_3)^2 - 4\eta \cdot k_1k_3
\]
which proves that one obtains infinitely many different homotopy types, when varying $k$. Note that the discussion on the non-Kählerness of the action in Section \ref{sec:nonkaehler} only cared about the signs of the $k_i$. Thus we have the following
\begin{prop}\label{p:lotsofhamt2 hell yeah}
There are infinitely many homotopy types among compact simply-connected $6$-dimensional manifolds which carry a Hamiltonian GKM $T^2$-action with $6$ fixed points but do not admit an invariant K\"ahler structure.
\end{prop}

\begin{rem}\label{rem:connectedstabs}
Note that the stabilizers of the produced examples are connected if and
only if all the $k_i$ are $\pm 1$. One can show that for any GKM fibration
$\Gamma\to B$ of twisted type over the GKM graph $B$ of $\CC \PP^2$ with $k_i=\pm
1$ the GKM graph $\Gamma$ is isomorphic to that of ${\mathrm{SU}}(3)/T^2$ or of
Eschenburg's twisted flag manifold ${\mathrm{SU}}(3)//T^2$, cf.\ Examples \ref{ex:flag} and \ref{ex:twflag}. Thus, among the infinitely many examples of Hamiltonian $T^2$-actions we
constructed for Proposition \ref{p:lotsofhamt2 hell yeah}, the only one with connected stabilizers and without invariant compatible K\"ahler structure is the original
Hamiltonian non-K\"ahler example due to Tolman. An example in which not
all $k_i$ are $\pm 1$ is given as follows:
\begin{center}
\begin{tikzpicture}
\draw[step=1, dotted, gray] (-5.5,-4.5) grid (3.5,4.5);

\draw[very thick] (1,3) -- ++(-5,0) -- ++(6,-6) -- ++(0,3) -- ++(-1,3);
\draw[very thick] (2,-3)--++(-1,+2)--++(0,4)++(0,-4)--++(-1,1)--++(2,0)++(-2,0)--++(-4,3);
  \node at (1,3)[circle,fill,inner sep=2pt]{};
%  \node at (1.3,1.3){$p_6$};

  \node at (-4,3)[circle,fill,inner sep=2pt]{};
%  \node at (-2.3,1.3){$p_1$};

  \node at (1,-1)[circle,fill,inner sep=2pt]{};
%\node at (1.3,-0.7){$p_2$};

  \node at (0,0)[circle,fill,inner sep=2pt]{};
%  \node at (0.3,0.3){$p_5$};

  \node at (2,0)[circle,fill,inner sep=2pt]{};
%  \node at (2.3,0.3){$p_3$};

  \node at (2,-3)[circle,fill,inner sep=2pt]{};
%  \node at (2.3,-3.3){$p_4$};

\end{tikzpicture}
\end{center}
\end{rem}

\bibliographystyle{acm}
%\bibliography{/home/pako/.config/TeXFiles/master.bib}
\bibliography{projectiveGKM}

\end{document}